\documentclass{amsart}
\usepackage{amsthm,amsfonts, amsmath,amssymb,tikz,calligra,mathrsfs}
\usetikzlibrary{matrix,arrows,decorations.pathmorphing}
\usepackage[all]{xy}
\usepackage{blkarray}

\theoremstyle{plain}
\newtheorem{theorem}{Theorem}[section]
\newtheorem{corollary}[theorem]{Corollary}

\newtheorem{lemma}[theorem]{Lemma}

\newtheorem{remark}[theorem]{Remark}

\newtheorem*{theorem*}{Theorem}

\theoremstyle{definition}
\newtheorem{definition}[theorem]{Definition}

\newcommand{\ds}{\oplus}
\newcommand{\tr}{\otimes}

\newcommand{\blf}[2]{\langle #1 \, , #2 \rangle }
\newcommand{\mor}{\longrightarrow}
\newcommand{\fun}{\mapsto}

\renewcommand{\d}[1]{\ensuremath{\mathbb{#1}}}
\renewcommand{\r}[1]{\ensuremath{\mathcal{#1}}}
\newcommand{\f}[1]{\ensuremath{\mathfrak{#1}}}

\newcommand{\frobdef}[2]{#1\stackrel{\operatorname{def}}{\dashrightarrow} #2}

\newcommand{\epi}{\xymatrix{{}\ar@{->>}[r]&{}}}

\def\Spec{\operatorname{Spec}}

\def\ch{\operatorname{char}}

\def\Hom{\operatorname{Hom}}
\def\Tor{\operatorname{Tor}}

\def\mod{\operatorname{mod}}
\def\ker{\operatorname{ker}}
\def\coker{\operatorname{coker}}
\def\im{\operatorname{im}}
\def\Id{\operatorname{Id}}
\def\rk{\operatorname{rk}}

\def\max{\operatorname{max}}
\newcommand{\mono}{\xymatrix{{}\ar@{^{(}->}[r]&{}}}
\def\gldim{gl.\operatorname{dim}}

\def\dim{\operatorname{dim}}

\DeclareMathOperator{\ShHom}{\mathscr{H}\text{\kern -2pt{\calligra\large om}}\,}
\DeclareMathOperator{\ShExt}{\mathscr{E}\text{\kern -2pt {\calligra\large xt}}\,}

\title[Some generalizations of Preprojective Algebras]{Some generalizations of Preprojective Algebras and their properties}
\author{LOUIS DE THANHOFFER DE VOLCSEY}
\thanks{The first author is a Ph.D student at UHasselt, Agoralaan, Hasselt, Belgium, \texttt{ldethanhoffer@me.com}}

\author{DENNIS PRESOTTO} 
\thanks{The second author is a Ph.D fellow with the FWO Flanders at UHasselt, Agoralaan, Hasselt, Belgium, \texttt{dennis.presotto@uhasselt.be}}
\begin{document}

\begin{abstract}
In this note we consider a notion of relative Frobenius pairs of commutative rings $S/R$. To such a pair, we associate an $\mathbb{N}$-graded $R$-algebra $\Pi_R(S)$ which has a simple description and coincides with the preprojective algebra of a quiver with a single central node and several outgoing edges in the split case. If the rank of $S$ over $R$ is 4 and $R$ is noetherian, we prove that $\Pi_R(S)$ is itself noetherian and finite over its center and that each $\Pi_R(S)_d$ is finitely generated projective. We also prove that $\Pi_R(S)$ is of finite global dimension if $R$ and $S$ are regular.
\end{abstract}

\maketitle

\tableofcontents

\section{Introduction}
\subsection{Definitions}
For the purposes of this paper, we consider pairs of commutative rings $R,S$ equipped with a map $R \mor S$. We often write such a pair as $S/R$. We will always assume $R$ is Noetherian, although some of the results also hold in higher generality.
\begin{definition}\label{def:frobenius}
We say that $S/R$ is \emph{relative Frobenius} of rank $n$ if:
\begin{itemize}
\item $S$ is a free $R$-module of rank $n$.
\item $\Hom_R(S,R)$ is isomorphic to $S$ as $S$-module.
\end{itemize}
\end{definition}
\begin{remark}\label{rem:defFrob}
\begin{itemize}
\item It is clear  that if $R$ is a field, a relative Frobenius pair coincides with a finite dimensional Frobenius algebra in the classical sense.
\item Let $e_1, \ldots, e_n$ be any basis for $S$ as an $R$-module. Then the second condition is equivalent to the existence of a $\lambda \in \Hom_R(S,R)$ such that the $R$-matrix $\left( \lambda(e_i e_j) \right)_{i,j}$ is invertible.
\item We may equally well assume that $S/R$ is projective of rank $n$. However all results we prove may be reduced to the free case by suitably localizing $R$.
\end{itemize}
\end{remark}

We shall need the following notation: for a relative Frobenius pair $S/R$, let $M:={}_R S_S$. This $R$-$S$-bimodule can be viewed as an $R\ds S$ bimodule by letting the $R$-component act on the left and the $S$-component on the right, the other actions being trivial. Similarly, we let $N:={}_S S_R$ and view it as an $R\ds S$-bimodule by only letting the $S$-component act on the left and the $R$-component act on the right, the other actions again begin trivial. We now define
\[ T(R,S):=T_{R\ds S}(M\ds N) \]
Note that by construction, we have $M\tr_{R\ds S} M=N \tr_{R\ds S} N=0$, hence
\[T(R,S)_2=\left( M_{R \ds S} N \right) \ds \left( N \tr_{R\ds S} M\right)=\left({}_R S\tr_S  S_R\right) \ds \left({}_S S\tr_R S_R\right) \]
The algebra we are interested in will be a quotient of $T(R,S)$ as follows: let $\lambda$ be a generator
of $\Hom_R(S,R)$ as an  $S$-module (this $\lambda$ exists by the Frobenius condition \ref{def:frobenius} we imposed). The $R$-bilinear form $\blf{a}{b}:=\lambda(ab)$ is clearly nondegenerate and hence we can find dual $R$-bases $(e_i)_i$, $(f_j)_j$  satisfying
\[ \lambda(e_if_j)=\delta_{ij} \]
\begin{definition} \label{def}
For a relative Frobenius pair, the \emph{generalized preprojective algebra} $\Pi_R(S)$  is given by
\[
T(R,S)/I
\]
where the ideal $I$ is generated by relations in degree 2 given by
\[
1\otimes 1\in {}_R S\otimes_S S_R
\]
\[
\sum_i e_i\otimes f_i\in {}_S S\otimes_R S_S
\]
\end{definition}
\begin{remark}
\noindent
Up to isomorphism, the above construction is independent of choice of generator and dual basis.
\end{remark}
\noindent
The name generalized preprojective algebra is motivated by the following:
\begin{lemma} \label{lem:classicpreprojective} Let $S$ be the ring $R^{\ds n}$.\\ 
Then $\Pi_R(S)$ is isomorphic to the preprojective algebra over $R$ associated to the quiver with one central vertex and $n$ outgoing arrows. 
\end{lemma}
\begin{proof}
Let $e_1, \ldots, e_n$ be the set of complete orthogonal idempotents in S and write $x_1, \ldots, x_n$ (respectively $y_1, \ldots, y_n) \in \Pi_R(S)_1$ for the corresponding elements in the bimodules $N$ (respectively $M$). We can describe the tensoralgebra $T(R,S)$ as the free algebra $F:=R \langle e_1, \ldots , e_n, x_1, \ldots , x_n, y_1, \ldots, y_n \rangle$ subject to the relations
\begin{enumerate}
\item $e_ie_j=\delta_{ij}e_i $.
\item  $e_ix_j=\delta_{ij}x_i$ and $y_i e_j=\delta_{ij}y_i$
\item $x_ie_j = e_iy_j=0$
\item $x_ix_j = y_iy_j=0$
\end{enumerate}The first relation defining $\Pi_R(S)$ is given by $1\tr 1 \in M \tr_S N$. The first unit is $1=\sum y_i$  whereas the second  is $1=\sum x_i$, we obtain
\begin{enumerate}
\setcounter{enumi}{4}
\item $y_1 x_1 + \ldots + y_n x_n= 0$
\end{enumerate}
To compute the second relation, we note that
\[ \lambda: S\rightarrow R: \sum_{i=1}^n r_ie_i \mapsto \sum_i r_i\]
is a generator of $\Hom_R(S,R)$ as an $S$-module and hence $(e_i)_i$ is a basis, self-dual for the associated form $\blf{}{}$ introduced in the discussion preceding Definition \ref{def} . The relation on ${}_S S \tr_R S_S$ now becomes
\begin{enumerate}
\setcounter{enumi}{5}
\item $x_1y_1 + \ldots + x_n y_n = 0$
\end{enumerate}
It now remains to show that the algebra $F$ subject to the above $6$ relations is isomorphic to the preprojective algebra of the quiver $Q$:
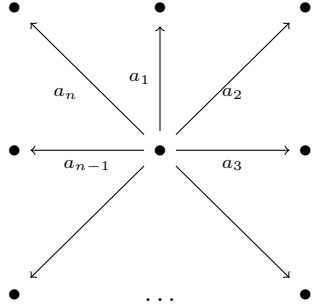
\begin{center}
\begin{tikzpicture}
\matrix(m)[matrix of math nodes,
row sep=4em, column sep=4em,
text height=1.5ex, text depth=0.25ex]
{\bullet  & \bullet & \bullet \\
\bullet  & \bullet & \bullet \\
\bullet & \ldots & \bullet \\};
\path[->,font=\scriptsize]
(m-2-2) edge node[auto]{$a_n$} (m-1-1)
        edge node[auto]{$a_1$} (m-1-2)
        edge node[below]{$a_2$} (m-1-3)
        edge node[below]{$a_3$}(m-2-3)
        edge (m-3-3)
        edge (m-3-1)
        edge node[auto]{$a_{n-1}$} (m-2-1)
;
\end{tikzpicture}
\end{center}
We let $\overline{Q}$ denote the formally doubled quiver of $Q$ and consider the map $F\mor R\overline{Q}$ defined  by
\begin{itemize}
\item sending $e_i$ to the outer node $n_i$
\item sending $y_i$ to the arrow $a_i$ and $x_i$ to the formal inverse $a_i^*$
\end{itemize}
The first 4 relations now precisely describe the multiplication in the path algebra of $\overline{Q}$ and the relations $(5)$ an $(6)$ precisely map to the two relations defining a preprojective algebra $\sum a_ia^*_i=0=\sum a^*_ia_i$
\end{proof}

\begin{remark}
In \cite{CBH}, Crawley-Boevey and Holland define so-called deformed preprojective algebras $\Pi^\lambda(Q)$ over a field $k$, where $Q$ is a quiver with $n$ vertices and $\lambda \in k^{\oplus n}$ is a parameter. They obtain a $k[\lambda]$-family which at the special point $\lambda = 0$ produces the classical preprojective algebra on $Q$. Our construction on the other hand produces an $R$-family of algebras of the form $\Pi_R(S)$ which in the smooth case produces the classical preprojective algebra at the geometric generic fibre. Hence we consider generalizations of preprojective algebras instead of deformations, this allows for certain special fibers to have infinite global dimension (if $S$ has infinite global dimension) but produces a nice family.
\end{remark}

\subsection{Statement of the Results}
We are particularly interested in the setting where $S/R$ is relative Frobenius of rank 4 (although a number of results are stated in higher generality). Moreover $R$ will always be a noetherian ring. We prove three basic properties of the algebra $\Pi_R(S)$ under these assumptions.\\[\medskipamount] 
\S 3 is dedicated to the following result:
\begin{theorem*}[see \ref{thm:free}]
The $R$-module $\Pi_R(S)_d$ is projective of rank $\left\{ \begin{array}{ll} 5(d+1) & \textrm{ if $d$ is even} \\
4(d+1) & \textrm{ if $d$ is odd} \end{array} \right.$
\end{theorem*}
In \S 4 we investigate the center of $\Pi_R(S)$. To this end, we use the short hand notation $Z_d(R,S) := Z(\Pi_R(S))_d$\begin{theorem*}[see \ref{lem:centersplit}, \ref{lem:basechangecenterrank4} and \ref{thm:centerprojectiverank2}]
$Z_d(R,S)$ is a split submodule of ${\Pi_R(S) }_d$ for each $d \in \mathbb{N}$, it is projective of rank:
\[ \rk(Z_d(R,S)) = \begin{cases} \frac{d}{4}+1 & \textrm{if } d \equiv 0 \ \mod \ 4 \\ \frac{d-2}{4} & \textrm{if } d \equiv 2 \ \mod \ 4 \\ 0 & \textrm{else} \end{cases} \]
We deduce from it that $Z(R,S)$ is compatible with base change under any morphism $R\mor R'$.
\end{theorem*}

\S  5 is dedicated to constructing a map
\[ \sigma_{R,S}: R[Z_4(R,S)]^{\oplus n} \mor \Pi_R(S) \]
and we prove
\begin{theorem*}[see \ref{thm:sigmamain} and \ref{thm:main}]
$\sigma_{R,S}$ is surjective, in particular $\Pi_R(S)$ is Noetherian and finite over its center.
\end{theorem*}

The final section covers the global dimension of $\Pi_R(S)$. Our main result there is:
\begin{theorem*}[see \ref{thm:globdim}]
If $R$ and $S$ have finite global dimension, then so does $\Pi_R(S)$. Moreover, we have the following explicit upper bound:
\[ gr.gl.dim(\Pi_R(S)) \leq max( gl.dim(R), gl.dim(S) ) + 2 \]
\end{theorem*}

Before proving the main theorems of this paper we explicitly describe the Frobenius algebras of rank 4 over an algebraically closed field and show that they are related by so-called Frobenius deformations (see Lemmas \ref{lem:classificationfrobenius} and \ref{lem:deformations}). The above theorems are then all proven using a common technique, namely we first prove them assuming $R$ is an algebraically closed field and $S$ is extremal in the deformation graph (\ref{eq:deformat}), which yields two specific cases. Then we extend the results step by step, increasing the generality of $R$ as follows (with references to the applied lemmas): 
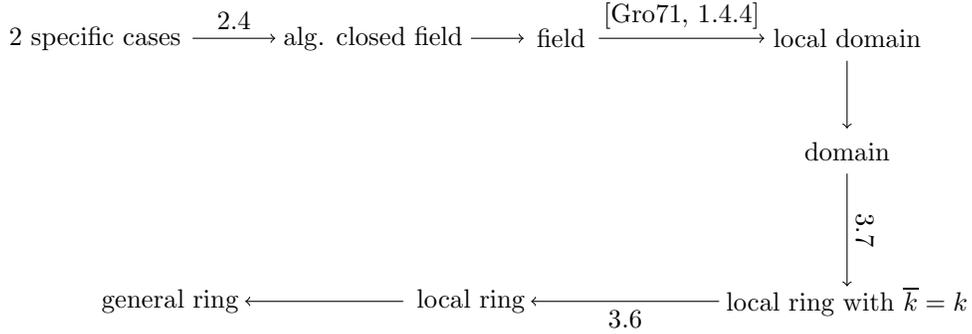
\begin{equation}
\begin{tikzpicture} \node at (0,1.5) {2 specific cases}; \node at (3.7,1.5) {alg. closed field}; \node at (6.2,1.5) {field}; \node at (10,1.5) {local domain}; \node at (10,0) {domain}; \node at (10,-2){local ring with $\overline{k}=k$}; \node at (5,-2) {local ring}; \node at (1,-2) {general ring}; \draw[->] (1.3,1.5) to  node[above]{\ref{lem:deformations}} (2.4,1.5); \draw[->] (5,1.5)--(5.7,1.5); \draw[->] (6.7,1.5) to node[above]{\cite[1.4.4]{GrothendieckSGA1}}(8.9,1.5); \draw[->] (10,1.2)--(10,0.3); \draw[->] (10,-0.3) to node[above,sloped]{\ref{lem:reducetodomain}} (10,-1.8); \draw[->] (8.3,-2) to node[below]{\ref{lem:localclosure}}(5.8,-2); \draw[->] (4.1,-2)--(2,-2); \end{tikzpicture} \label{diag:reduction}
\end{equation}

\section*{Acknowledgements}
The authors thank Michel Van den Bergh for providing many interesting ideas and for pointing out the Morita equivalence as in Lemma \ref{lem:dihedralskew}. The authors thank Johan de Jong for suggesting the use of universal families of Frobenius extensions (see e.g. Lemma \ref{lem:reducetodomain}). The authors also thank Kevin De Laet for helpful discussions.

\section{Preliminaries}
Throughout the paper we concern ourselves mostly with Frobenius pairs $S/R$ of rank 4. If $R$ is a field, this implies that $S$ is a Frobenius algebra of dimension 4. It is an easy exercise to describe all such algebras in the algebraically closed case:
\begin{lemma}\label{lem:classificationfrobenius}
Let $k$ be an algebraically closed field and $F$ a commutative Frobenius algebra of dimension 4 over $k$. Then $F$ is isomorphic to one of the following algebras:
$$\left\{
\begin{array}{cc}
&k\ds k\ds k\ds k\\ 
&k[t]/(t^2)\ds k\ds k\\
&k[s]/(s^2)\ds k[t]/(t^2)\\
&k[t]/(t^3)\ds k\\
&k[t]/(t^4)\\
&k[s,t]/(s^2,t^2)
\end{array}
\right.
$$

\end{lemma}

\begin{proof}
First recall that 
\begin{enumerate}
\item a direct sum of Frobenius algebras is itself Frobenius.
\item a finite dimensional commutative local $k$-algebra is Frobenius if and only if it has a unique minimal ideal.
\end{enumerate}
It follows immediately that $k[t]/(t^n)$ is Frobenius (of dimension $n$) over $k$ as it has a unique minimal ideal $(t^{n-1})$ and $k[x_1, \ldots , x_n]/(x_1^2, \ldots, x_n^2)$ is also Frobenius (of dimension $2^n$) with unique minimal ideal $(x_1 \cdot \ldots \cdot x_n)$. Thus the algebras in the above list are certainly Frobenius.\\
Now let $F$ be Frobenius of dimension $4$. Since $F$ is Artinian, the structure theorem for Artinian rings \cite[Theorem 8.7]{atiyahmacdonald} states that $F$ must (uniquely) decompose as a direct sum of local, Artinian $k$-algebras:
\[ F\cong F_1\ds \ldots \ds F_n \]
We can now use the classification of local $k$-algebras of small rank in \cite[Table 1]{poonen}.\\
If $n=4$, then clearly $F=k\ds k\ds k\ds k$.\\
If $n=3$, then $F\cong A_1\ds k \ds k$ where $\dim_k(A_1)=2$, hence $A_1\cong k[t]/(t^2)$ which is Frobenius.\\
If $n=2$, then either $F$ splits as a sum of $2$-dimensional local $k$-algebras, in which case we obtain $F\cong k[s]/(s^2)\ds k[t]/(t^2)$ or $F=A_1\ds k$ where $\dim_k(A_1)=3$. This again yields 2 possibilities: either $A_1 \cong k[t]/(t^3)$, which is Frobenius, or $A_1 \cong k[s,t]/(s,t)^2$. The latter however cannot be Frobenius as it is not self-injective (the morphism $A_1t \mor A_1: t \mapsto s$ cannot be lifted to $A_1 \mor A_1$).
\\
Finally, assume $n=1$. In this case $F$ is a local $k$-algebra of dimension $4$ and by \cite{poonen} takes one of the five following forms:
\[ \begin{cases} k[t]/(t^4) \\ k[s,t]/(s^2,t^2) \\ k[s,t]/(s^2,st,t^3) \\ k[s,t,u]/(s,t,u)^2 \\ k[s,t]/(s^2+t^2,st) & \textrm{(if char($k$)=2)}\end{cases} \]
The first two algebras are Frobenius whereas the other three are not as they are not self-injective by a similar argument as above.\\
\end{proof}
The 6 Frobenius algebras listed in the above lemma are related to each other by a notion closely related to deformations. For this purpose, we introduce the following ad hoc notion:
\begin{definition}\label{def:frobeniusdeformation}
Let $F$ and $G$ be Frobenius algebras over a field $k$.\\
A \emph{Frobenius deformation} of $F$ to $G$  is a $k[[u]]$-algebra $D$ such that $D$ is relatively Frobenius over $k[[u]]$  and
\begin{enumerate}
\item $D/uD \cong F$ as  a $k$-algebra
\item $D_{(u)}  \cong G\tr_k k((u))$ as a $k((u))$-algebra
\end{enumerate}
we write $\frobdef{F}{G}$
\end{definition}
\begin{remark}
Instead of requiring that $D/k[[u]]$ is relative Frobenius we may equivalently require that $D$ is free over $k[[u]]$ with rank equal to the dimension of $F$. The condition that $\Hom_{k[[u]]}(D,k[[u]])$ should be isomorphic to $D$ as $D$ modules is immediate by the corresponding condition on $F/k$.
\end{remark}
\begin{lemma}\label{lem:deformations}
There is a diagram of Frobenius deformations

\begin{eqnarray}
\begin{tikzpicture}
\label{eq:deformat}
\matrix(m)[matrix of math nodes,
row sep=3em, column sep=2em,
text height=1.5ex, text depth=0.25ex]
{ & & k[t]/(t^2) \oplus k[s]/(s^2) & & \\
k[s,t]/(s^2,t^2) & k[t]/(t^4) & & k[t]/(t^2) \oplus k\oplus k & k^{\oplus 4}\\
 & & k[t]/(t^3) \oplus k & & \\};
\path[->,dashed,font=\scriptsize]
(m-2-2) edge node[above,sloped]{$\operatorname{def}$} node[below]{2} (m-1-3)
        edge node[above,sloped]{$\operatorname{def}$} node[below]{3} (m-3-3)
(m-1-3) edge node[above,sloped]{$\operatorname{def}$} node[below]{4} (m-2-4)
(m-3-3) edge node[above,sloped]{$\operatorname{def}$} node[below]{5} (m-2-4)
(m-2-4) edge node[above]{$\operatorname{def}$} node[below]{6} (m-2-5)
(m-2-1) edge node[above]{$\operatorname{def}$} node[below]{1} (m-2-2);
\end{tikzpicture}
\end{eqnarray}

\end{lemma}

\begin{proof}
We first describe $\frobdef{F := k[s,t]/(s^2,t^2)}{G := k[t]/(t^4)}$. Let $R:=k[[u]]$, $K:=k((u))$ and define
\[ D:=R[s,t]/(us-t^2,s^2,t^4)\] We claim that $D$ defines a deformation from $F$ to $G$. It is clear that $D/uD \cong F$ as a $k$-algebra and the map 
\[ D \mor K[t]/(t^4): u \mapsto u, s \mapsto t^2/u, t \mapsto t\] factors through an isomorphism
\[  D_{(u)} \mor K[t]/(t^4) = G \otimes_k K \]
Hence by the above remark it suffices to check that $D$ is a free $R$-module of rank 4. This is obviously the case with $e_1=1, e_2=s, e_3=t, e_4=st$ providing an $R$-basis for $D$.

The other cases are similar. We first use the Chinese Remainder Theorem to find an alternate presentation for $F$ of the forms $k[t]/(f(t))$. Then for each deformation $\frobdef{F}{G}$, we try to find an alternate presentation for $G \otimes_k K$ (again using the Chinese Remainder Theorem) of the form $K[t]/(g(t))$ in such a way that $g(t)\arrowvert_{u=0}=f(t)$. We then exhibit an $R$-algebra $D:=R[t]/(g(t))$. We leave the reader to check that in each of our choices, $(1,t,t^2,t^3)$ defines an $R$-basis
.

\begin{center}
\begin{tabular}{|c|c|c|c|c|c|}
\hline
 & & & & & \\
number & 2 & 3 & 4 & 5 & 6* ($char(k)\neq 2$) \\
 & & & & & \\
\hline
 & & & & & \\
$g(t)$ & $t^2(t-u)^2$ & $t^3(t-u)$ & $(t-1)^2t(t-u)$ & $t^2(t-1)(t-u)$ & $(t^2-u^2)(t^2-1)$ \\
 & & & & & \\
\hline
\end{tabular}
\end{center}
* If $\ch(k)=2$, one chooses $D=R[t]/(t(t-u)) \oplus R^{\oplus 2}$ for deformation nr. 6 In this case $(1,0,0), (t,0,0), (0,1,0), (0,0,1)$ provides the required $R$-basis for $D$.
\end{proof}

\section{Computing $\rk(\Pi_R(S)_d)$}
The construction of the algebra $\Pi_R(S)$ (recall the definition from \ref{def}) is compatible with base change in the following way:

\begin{lemma}[Base Change for $\Pi_R(S)$] \label{lem:basechange}
Let $S/R$ be relative Frobenius of rank $n$ and $R\mor R'$ a morphism of rings. Then
\begin{enumerate}
\item
$(R'\tr_R S)/R'$ is relative Frobenius of rank $n$
\item there is a canonical isomorphism $$R'\tr_R \Pi_R(S) \cong \Pi_{R'}(R'\tr_R S)$$ 
\end{enumerate}
\end{lemma}
\begin{proof}
Assume that $S/R$ is relative Frobenius. Then we can pick an $R$-basis $\{ e_1, \ldots , e_n\}$ for $S$ and a generator $\lambda$ for the $S$-module $\Hom_R(S,R)$. It is then easy to see that $\{ 1\otimes e_1, \ \ldots , \ 1 \otimes e_n \}$ is an $R'$-basis for $R'\tr_R S$ and that $1 \otimes \lambda$ is a generator for the $S'\tr_R S$-module $\Hom_{R'}(R'\tr_R S,R'$), proving the first point. With this data we can thus construct $\Pi_{R'}(R' \otimes_R S)$. Moreover,
\[ R'\tr_R\left( {}_R S_S\ds {}_S S_R\right) \cong {}_{R'} (R' \tr_R S)_{R'\tr_R S}\ds {}_{R'\tr_R S} (R'\tr_R S)_{R'} \] as an $(R',R'\tr_R S)$-bimodule, and we obtain a canonical isomorphism
\[ R'\tr_R T(R,S) \cong T(R',R'\tr_R S) \]
which by our choice of basis preserves the relations, inducing an isomorphism
\[ R'\tr_R \Pi_R(S) \cong \Pi_{R'}(R'\tr_R S) \qedhere \]
\end{proof}

To prove that the $R$-modules $\Pi_R(S)_d$ are projective and to compute their ranks, following the method of proof described by diagram (\ref{diag:reduction}) in the introduction, we first treat the case where $R$ is an algebraically closed field. We have the following lemma relating these vector spaces under deformation:
\begin{lemma}\label{lem:ineqpirs}
Let $F$ and $G$ be Frobenius algebras over $k$ and let $\frobdef{F}{G}$ be a Frobenius deformation. Then for all $d$, we have
\[ \dim_k(\Pi_k(F)_d) \geq \dim_k(\Pi_k(G)_d) \]
\end{lemma}

\begin{proof}
Let $R=k[[u]]$ and $K=k((u))$.\\
Let $m=\dim_k (\Pi_k(F)_d)$.
Assume that $D$ is the $R$-algebra deforming $F$ to $G$. Then since 
$$\Pi_k(F)=\Pi_k(k\tr_R D)=k\tr_R \Pi_R(D)$$ 
by Lemma \ref{lem:basechange}, Nakayama's lemma implies that a $k$-basis of length $m$ for $\Pi_k(F)_d$ lifts to a set of generators for $\Pi_R(D)_d$. Moreover, as 
$$K\tr_k \Pi_k(G)=\Pi_K(K\tr_k G)=\Pi_K(K\tr_R D)=K\tr_R(\Pi_R(D))$$
this set of generators contains a $K$-basis for  $K\tr \Pi_k(G)$. It follows that 
\[ \dim_K(K\tr_k(\Pi_k(G)_d)=\dim_k(\Pi_k(G)_d) \le m \qedhere \]
\end{proof}
We will now prove that in the case of Frobenius algebras of rank 4 the above inequality is actually an equality. We first compute the ranks in two explicit cases:
\begin{lemma}\label{lem:bikwaddegree}
We have
\[ \dim_k \left( \Pi_k \left(\frac{k[s,t]}{(s^2,t^2)} \right)_d \right) \leq \left\{ \begin{array}{cl} 5(d+1) & \textrm{if $d$ is even}\\
4(d+1) & \textrm{if $d$ is odd} \end{array} \right. \]
\end{lemma}

\begin{proof}
This is proven in appendix $A.1$.
\end{proof}
\begin{lemma} \label{lem:etingoff}
Let $k$ be an algebraically closed field, then
\[ \dim_k \Big( \Pi_k(k^{\oplus 4})_d \Big) = \left\{ \begin{array}{cl} 5(d+1) \textrm{ if $d$ is even} \\
4(d+1) \textrm{ if $d$ is odd} \end{array} \right. \]
\end{lemma}
\begin{proof}
By Lemma \ref{lem:classicpreprojective}, $\Pi_k(S)$ is the preprojective algebra over $k$ associated to the extended Dynkin quiver of $Q = \widetilde{D_4}$.

Let $\overline{Q}$ denote the formally doubled quiver, let 0 denote the central vertex and  1, 2, 3, 4 the outer vertices. Then for each $d \in \mathbb{N}$ we consider the matrix $W_d \in \mathbb{N}^{5 \times 5}$ where $(W_d)_{ij}$ is the number of paths  of length $d$ in $\overline{Q}$ starting at vertex $i$ and ending at vertex $j$, modulo relations. Finally write $W(t) = \sum_{d=0}^\infty W_d t^d \in \mathbb{N}^{5 \times 5}[[t]]$. Then by \cite[Proposition 3.2.1]{Etingof} we have
\[ W(t) = \frac{1}{1-t \cdot C + t^2} \]
Where $C$ is the adjacency matrix of $\overline{Q}$, i.e.
\begin{eqnarray*}
W(t) & = & \left ( 1 - t \cdot \left( \begin{array}{ccccc} 0 & 1 & 1 & 1 & 1 \\ 1 & 0 & 0 & 0 & 0 \\ 1 & 0 & 0 & 0 & 0 \\ 1 & 0 & 0 & 0 & 0 \\ 1 & 0 & 0 & 0 & 0 \end{array} \right) +t^2 \right)^{-1} \\
 & = & \frac{1}{(1-t^2)^2(1+t^2)} \cdot \left( \begin{array}{ccccc} (1+t^2)^2 & t(1+t^2) & t(1+t^2) & t(1+t^2) & t(1+t^2) \\
 t(1+t^2) & 1-t^2+t^4 & t^2 & t^2 & t^2 \\ t(1+t^2) & t^2 & 1-t^2+t^4 & t^2 & t^2 \\ t(1+t^2)& t^2 & t^2 & 1-t^2+t^4 & t^2 \\ t(1+t^2) & t^2 & t^2 & t^2 & 1-t^2+t^4 \end{array} \right)
\end{eqnarray*}
This gives the desired result as the Hilbert series of $\Pi_k(S)$ now becomes
\begin{eqnarray*}
h_{\Pi_k(S)}(t) & = & \sum_{d=0}^{\infty} \left( \sum_{i,j = 0}^4 (W_d)_{i,j} \right) t^d \\
 & = & \sum_{i,j = 0}^4 \sum_{d=0}^{\infty} (W_d)_{i,j} t^d \\
 & = & \frac{(1+t^2)^2 + 8t(1+t^2) + 4 (1-t^2+t^4) + 12 t^2}{(1-t^2)^2(1+t^2)} \\
 & = & \frac{5+8t+5t^2}{(1-t^2)^2} \\
 & = & (5+8t+5t^2) \sum_{l=0}^\infty (l+1)t^{2l} \\
 & = & \sum_{l=0}^\infty (5l+5(l+1))t^{2l} + 8(l+1)t^{2l+1} \\ 
 & = & \sum_{l=0}^\infty (5(2l+1))t^{2l} + 4((2l+1)+1)t^{2l+1}
\end{eqnarray*}
\end{proof}

\begin{corollary}\label{cor:freefield}
Let $k$ be a field and $F$ a Frobenius algebra (of rank 4) over $k$.\\ Then the dimension of $\Pi_k(F)_d$ is given by
\[ \dim_k \left( \Pi_k(F)_d \right) = \left\{ \begin{array}{cl} 5(d+1) \textrm{ if $d$ is even} \\
4(d+1) \textrm{ if $d$ is odd} \end{array} \right. \]
\end{corollary}
\begin{proof}
By Lemma \ref{lem:basechange} we can reduce to the case where $k$ is algebraically closed. The statement then follows as a combination of Lemmas \ref{lem:classificationfrobenius}, \ref{lem:deformations}, \ref{lem:ineqpirs}, \ref{lem:bikwaddegree} and \ref{lem:etingoff}
\end{proof}
To extend the result from fields to general rings we will need the following two lemmas. They essentially show that locally every relative Frobenius pair can be obtained through base change  (following Lemma \ref{lem:basechange}) from a relative Frobenius pair where the ground ring is a polynomial ring over the integers.
\begin{lemma}\label{lem:localclosure}
Let $R$ be a local ring with residue field $k$ with algebraic closure $\overline{k}$. Then there is a faithfully flat morphism $R \mor \overline{R}$ where $\overline{R}$ is a local ring with residue field $\overline{k}$.
\end{lemma}
\begin{proof} This is an immediate application of \cite[10.3.1]{GrothendieckEGA3}
\end{proof}

\begin{lemma}\label{lem:reducetodomain}
Let $R$ be a local ring with an algebraically closed residue field $k$. Let $S/R$ be relative Frobenius of rank 4. Then there exists a domain $\tilde{R}$, together with a morphism $\tilde{R} \mor R$ and a ring $\tilde{S}$ with $\tilde{S}/\tilde{R}$ relative Frobenius of rank 4 such that $\tilde{S} \otimes_{\tilde{R}} R \cong S$.\\
Moreover $\tilde{R}$ can be chosen to be chosen of the form $\d{Z}[x_1, \ldots, x_m]_f$, the localization of a polynomial ring over $\d{Z}$ at some non-zero element $f$.
\end{lemma}
\begin{proof}
We prove the theorem in a specific case and quickly sketch the other cases, leaving some details to the reader. By Lemmas \ref{lem:basechange} and \ref{lem:deformations}, $S \otimes_R k$ is one of $6$ Frobenius algebras. Assume $S\otimes k =k[s,t]/(s^2,t^2)$ and let $\tilde{s},\tilde{t} \in S$ be lifts of $s$ and $t$. Since $(1,s,t,st)$ is a $k$-basis for $S\tr_R k$ by  the proof of \ref{lem:deformations}, Nakayama's lemma implies that $(1,\tilde{s},\tilde{t},\tilde{s}\tilde{t})$ forms a set of $R$-generators for $S$. 
In particular we can write:
\begin{eqnarray*}
\tilde{s}^2 & = & a_1 + b_1 \tilde{s} + c_1 \tilde{t} + d_1 \tilde{s} \tilde{t} \\
\tilde{t}^2 & = & a_2 + b_2 \tilde{s} + c_2 \tilde{t} + d_2 \tilde{s} \tilde{t}
\end{eqnarray*}
where $a_1, \ldots , d_2$ all lie in the maximal ideal of $R$ (because $s^2=t^2=0$ in $S \otimes k$). We thus have a canonical morphism 
\[ \pi: R[\tilde{s},\tilde{t}]/(a_1 + b_1 \tilde{s} + c_1 \tilde{t} + d_1 \tilde{s} \tilde{t} - \tilde{s}^2, a_2 + b_2 \tilde{s} + c_2 \tilde{t} + d_2 \tilde{s} \tilde{t} - \tilde{t}^2)\mor S \] such that $\pi\tr_R k$ is the identity morphism. It follows that $\pi$ is surjective, moreover since $S$ is free over $R$, we have $0= \ker(\pi\tr_R k)=\ker(\pi)\tr_R k$ and $\ker(\pi)=0$ using Nakayama's lemma once more. $\pi$ is thus an isomorphism.\\
There is a canonical morphism
\[ A:=\d{Z}[a_1,b_1,c_1,d_1,a_2,b_2,c_2,d_2] \mor R \]
Let $f = 1 - d_1 d_2$ and denote $\tilde{R}=A_f$. Then as the image of $f$ in $R$ is invertible (because $d_1,d_2$ lie in the maximal ideal of $R$), the above morphism factors through a morphism $\tilde{R} \mor R$. Finally set 
\[\tilde{S} = \tilde{R}[\tilde{s},\tilde{t}]/(a_1 + b_1 \tilde{s} + c_1 \tilde{t} + d_1 \tilde{s} \tilde{t} - \tilde{s}^2, a_2 + b_2 \tilde{s} + c_2 \tilde{t} + d_2 \tilde{s} \tilde{t} - \tilde{t}^2).\] By construction we have 
\[ \tilde{S} \otimes_{\tilde{R}} R \cong R[\tilde{s},\tilde{t}]/(a_1 + b_1 \tilde{s} + c_1 \tilde{t} + d_1 \tilde{s} \tilde{t} - \tilde{s}^2, a_2 + b_2 \tilde{s} + c_2 \tilde{t} + d_2 \tilde{s} \tilde{t} - \tilde{t}^2) \stackrel{\pi}{\cong}S \] 
It hence suffice to prove $\tilde{S}/\tilde{R}$ is relative Frobenius of rank 4. For this note that $(e_i)_{1,\ldots, 4}:=(1,\tilde{s},\tilde{t},\tilde{s}\tilde{t})$ is an $\tilde{R}$-basis for $\tilde{S}$. Moreover, if we let $\lambda \in \Hom_{\tilde{R}}(\tilde{S},\tilde{R})$ denote the projection onto the component $\tilde{R}\tilde{s}\tilde{t}$, the matrix of $\lambda(e_i.e_j)$ is of the form
\[\Theta= \begin{bmatrix}  0&0&0&1\\
0&d_1&1&*\\
0&1&d_2&*\\
1&*&*&* \end{bmatrix}\] 
Hence $\Theta$ has determinant $1-d_1 d_2$, which by construction is invertible in $\tilde{R}$, proving that $\tilde{S}$ is indeed Frobenius of rank $4$ over $\tilde{R}$ by Remark \ref{rem:defFrob}.\\ \\ 
In the 5 other cases from Lemma \ref{lem:classificationfrobenius} we have $S \otimes k = k[t]/(t^4 + a t^3 + bt^2 +ct +d)$ for some $a,b,c,d \in k$ and we can choose $\tilde{R}, \tilde{S}$ of the form $\tilde{R}:= \d{Z}[\alpha,\beta,\gamma,\delta]$ and $\tilde{S} := \tilde{R}[t]/(t^4 + \alpha t^3 + \beta t^2 + \gamma t + \delta)$. For each choice of $\alpha, \beta, \gamma, \delta$ we have that $\tilde{S}/\tilde{R}$ is relative Frobenius of rank 4, because the corresponding matrix $\Theta$ will have determinant exactly 1. We leave the details to the reader.
\end{proof}

We can now prove the main theorem of this section:

\begin{theorem}\label{thm:free}
$\Pi_R(S)_d$ is projective of rank $\begin{cases} 5(d+1) & \textrm{ if $d$ is even} \\ 4(d+1) & \textrm{if $d$ is odd} \end{cases}$. 
\end{theorem}

\begin{proof}
First let $R$ be a local domain with residue field $k$ and field of fractions $K$. By Corollary \ref{cor:freefield} and Lemma \ref{lem:basechange} we have for each degree $d$:
\[ \dim_K(K\tr_R \Pi_R(S)_d) = \dim_K(\Pi_K(K \tr_R S)_d) =  \dim_k(\Pi_k(k\tr_R S)_d)= \dim_k (k \tr_R \Pi_R(S)_d) \]
Then \cite[1.4.4]{GrothendieckSGA1} implies that $\Pi_R(S)$ is free of the stated ranks computed in Corollary \ref{cor:freefield}.\\[\medskipamount]
Next, let $R$ be any domain. Then  for each $\f{p}\in \Spec(R)$, $R_\f{p}\tr_R \Pi_R(S) \cong \Pi_{R_\f{p}}(R_{\f{p}} \otimes S)$ is a generalized preprojective algebra over the local domain $R_{\f{p}}$ and hence in each degree is a free module of the stated rank. As these ranks do not depend on the choice of $\f{p}$, Serre's theorem (see for example \cite{serre}) now implies that ${\Pi_R(S)}_d$ is projective of the stated rank.\\[\medskipamount]
Next, let $R$ be a (possibly non-reduced) local ring with algebraically closed residue field. Then by Lemma \ref{lem:reducetodomain} there is a domain $\tilde{R}$, a morphism $\tilde{R} \mor R$ and a ring $\tilde{S}$ such that $\tilde{S}/\tilde{R}$ is relative Frobenius of rank 4 and $S \cong \tilde{S} \otimes_{\tilde{R}} R$. By the above $\Pi_{\tilde{R}}(\tilde{S})_d$ is a projective $\tilde{R}$-module of the given ranks and hence $\Pi_R(S)_d = \Pi_{\tilde{R}}(\tilde{S})_d \otimes R$ is a projective $R$-module of the above rank.\\[\medskipamount]
To extend the result to general local rings, we invoke Lemma \ref{lem:localclosure} to find a faithfully flat morphism $\phi: R \mor \overline{R}$. By the above $\Pi_{\overline{R}}(\overline{R} \otimes S)_d \cong \overline{R} \otimes \Pi_R(S)_d$ is a free $\overline{R}$-module of the desired rank. By the faithfully flatness of $\phi$, $\Pi_R(S)_d$ is itself a free $R$-module of the desired rank.\\[\medskipamount]
Finally we extend the statement from local rings to general commutative rings by again applying Serre's theorem \cite{serre} .
\end{proof}

The following Lemma is a slightly more technical variation of Theorem \ref{thm:free} which will be required in the final section of this paper.

\begin{lemma} \label{lem:splithilbertseries}
$(1_R \cdot \Pi_R(S))_d$ and $(1_S \cdot \Pi_R(S))_d$ are projective $R$-modules of ranks 
\[ \rk((1_R \cdot \Pi_R(S))_d=\begin{cases} d+1 & \textrm{if $d$ is even} \\ 2(d+1) & \textrm{if $d$ is odd} \end{cases} \]
and
\[ \rk((1_S \cdot \Pi_R(S))_d=\begin{cases} 4(d+1) & \textrm{if $d$ is even} \\ 2(d+1) & \textrm{if $d$ is odd} \end{cases} \]
\end{lemma}
\begin{proof}
We can write $\Pi_R(S) = 1_R \cdot \Pi_R(S) \oplus 1_S \cdot \Pi_R(S)$. This immediately implies that both modules are projective by Theorem \ref{thm:free}. Moreover, it is easy to see that this decomposition is preserved under both base change through a morphism $R\mor R'$  and Frobenius deformations $\frobdef{F}{G}$ in the obvious sense. From this, we can conclude, using an argument similar to the proof of Theorem \ref{thm:free} shows that it suffices to check the cases where $R=k$ and  $S=k^{\oplus 4}$ or $S=k[s,t]/(s^2,t^2)$.\\[\medskipamount] 
For the first case we notice that the values of the Hilbert series $h_{1_k \cdot \Pi_R(S)}(t)$ can be computed using the proof of Lemma \ref{lem:etingoff} by adding the entries in the first column of $W(t)$, giving 
\begin{eqnarray*} h_{1_k \cdot \Pi_k(S)}(t) & = & \frac{(1+t^2)^2 + 4\cdot t (1+t^2)}{(1-t^2)^2(1+t^2)} \\
& = & \frac{1+ 6t + t^2}{(1-t^2)^2} \\
& = & (1+6t+t^2) \sum_{l=0}^\infty (l+1)t^{2l} \\
& = &  \sum_{l=0}^\infty (2l+1)t^{2l} + \sum_{l=0}^\infty 2((2l+1)+1)t^{2l+1} \end{eqnarray*}
In a similar fashion, we find 
\begin{eqnarray*} h_{1_S \cdot \Pi_k(S)}(t) & = & \sum_{l=0}^\infty 4(2l+1)t^{2l} + \sum_{l=0}^\infty 2((2l+1)+1)t^{2l+1} \end{eqnarray*}
For the second case where we assume $S=k[s,t]/(s^2,t^2)$ this is a dreary calculation which follows from  the ``Type I''-``Type II''-classification of the generators of $\Pi_k(S)$ found in appendix $A.1$.
\end{proof}


\section{Base Change for $Z(\Pi_R(S))$ and $\rk(Z(\Pi_R(S))_d)$}
Throughout this section, $S$ will be relative Frobenius of rank $4$ over a noetherian ring $R$. In this section, we will prove some results describing the center of $\Pi_R(S)$. To ease notation, we will write $Z_d(R,S)$ for the degree $d$-part of the center of $\Pi_R(S)$.

\begin{theorem} \label{lem:centersplit}
$Z_d(R,S)$ is a split $R$-submodule of ${\Pi_R(S)}_d$ for each $d \in \mathbb{N}$.
\end{theorem}
\begin{theorem}\label{lem:basechangecenterrank4}
Let $R\mor R'$ a morphism of rings. Then the canonical  morphism
\[ Z(\Pi_R(S))\otimes_R R'\mor Z(\Pi_{R'}(S\otimes_R R'))\]
is an isomorphism.
\end{theorem}
\begin{theorem} \label{thm:centerprojectiverank2}
$Z_d(R,S)$ is a projective $R$-module of rank

\[ \begin{cases} \frac{d}{4}+1 & \textrm{if } d \equiv 0 \ (mod \ 4) \\ \frac{d-2}{4} & \textrm{if } d \equiv 2 (mod \ 4) \\ 0 & \textrm{else} \end{cases}\]

\end{theorem}

The proofs of these theorems are heavily intertwined, we shall prove them using to the following diagram of implications:
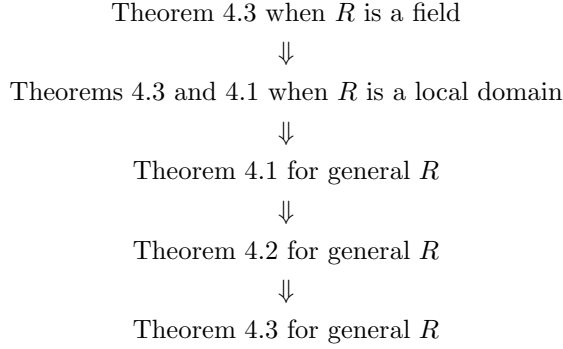
\begin{eqnarray*}
 & \textrm{Theorem \ref{thm:centerprojectiverank2} when $R$ is a field} & \\
 & \Downarrow & \\
 & \textrm{Theorems \ref{thm:centerprojectiverank2} and \ref{lem:centersplit} when $R$ is a local domain} & \\ 
 & \Downarrow & \\
 & \textrm{Theorem \ref{lem:centersplit} for general $R$} & \\  
 & \Downarrow & \\
 & \textrm{Theorem \ref{lem:basechangecenterrank4} for general $R$} & \\   
 & \Downarrow & \\
 & \textrm{Theorem \ref{thm:centerprojectiverank2} for general $R$} & \\  
\end{eqnarray*}
In several of these steps we use the fact that in each degree the center $Z_d(R,S)$ can be obtained as kernel of a morphism between projective $R$-modules. For this recall from the discussion preceding \ref{def} that since $\Pi_r(S)_0=R\ds S$  and $S$ is a free $R$-module of rank 4. there exist an $R$-basis $\{1_R:=a^0_1 \ldots ,a^0_5\}$ for $\Pi_R(S)_0$. Moreover, there exists an $R$-basis of the form 
\[(a_i^1)_{i=1,\ldots, 8}:=\{e_1 ,\ldots,e_4,f_1 ,\ldots,f_4\}\] for $\Pi_R(S)_1$ such that $\lambda (e_if_j)=\delta_{ij}$ for some chosen generator $\lambda$ of the $S$-module $\Hom_R(S,R)$. Now, since $\Pi_R(S)$ is generated in degrees $0$ and $1$, for each degree $d$ there is a morphism
\begin{equation}\label{eq:center}
\phi_{R,S}: \Pi_R(S)_d \mor (\Pi_R(S)_d^{\oplus 5}) \bigoplus (\Pi_R(S)_{d+1}^{\oplus 8}): x\fun \Big( \left([x,a^0_i]\right)_i,\left( [x,a^1_j]\right)_j \Big)
\end{equation}
whose kernel is precisely $Z_d(R,S)$. In other words  there is a left-exact sequence of the form:
\begin{equation}\label{eq:center4}
0 \mor Z_d(R,S) \mor \Pi_R(S)_d \stackrel{\phi_{R,S}}{\mor} \Pi_R(S)_d^{\oplus 5} \ds \Pi_R(S)_{d+1}^{\oplus 8}
\end{equation}

In particular we obtain the following special case of Theorem \ref{lem:basechangecenterrank4}:
\begin{lemma}[flat base change]
\label{lem:flatbasechange}
Let $R\mor R'$ be a flat morphism of rings. Then the canonical map
$$R'\tr_R Z_d(R,S) \mor Z_d(R'\tr_R S)$$
is an isomorphism for each $d \in \mathbb{N}$.
\end{lemma}

\begin{proof}
The construction of the morphism $\phi_{R,S}$ is compatible with base change by the proof of Lemma \ref{lem:basechange} and tensoring with flat modules preserves left exact sequences. Hence
\[ R' \otimes Z_d(R,S) = R' \otimes \ker(\phi_{R,S}) = \ker(R' \otimes \phi_{R,S}) = \ker ( \phi_{R', R' \otimes S}) = Z_d(R', R' \otimes S) \]
\end{proof}
As stated in Theorem \ref{lem:basechangecenterrank4} we will show that in fact the above result holds for \emph{arbitrary} morphisms. Following the technique of proof outlined in diagram (\ref{diag:reduction}) we shall first compute the dimension of $Z_d(R,S)$ in two specific cases:

\begin{lemma}\label{lem:dihedralskew}
Let $k$ be an algebraically closed field with $\ch(k)\neq 2$ and $F = k^{\oplus 4}$, then $\Pi_k(F)$ is Morita equivalent to  the skew group ring $k[x,y] \# BD_2$ where 
\[ BD_2 = \ < a,b \mid a^4=b^4=1, a^2=b^2, ab=ba^3 > \]
is the binary dihedral group of order 8 acting on $k[x,y]$ via
\[ a \cdot x = ix, \ a \cdot y = -iy, \ b \cdot x = y, \ b \cdot y = x \]
\end{lemma}
\begin{proof}
Let $Q$ be be the extended Dynkin quiver $\widetilde{D_4}$ and $\overline{Q}$ the formally doubled quiver. Then $\overline{Q}$ is the McKay-quiver of $BD_2$ acting on $k[x,y]$ through the rule described above. Now, by \cite[Corollary 4.2]{Preprojectivemorita} (which was already announced in \cite{ReitenVDB}) the (classical) preprojective algebra $\Pi_k(Q)$ is Morita equivalent to $k[x,y] \# BD_2$. The result now follows from Lemma \ref{lem:classicpreprojective}.
\end{proof}

\begin{lemma}\label{lem:dihedralcenter}
Let $k$ be an algebraically closed field with $\ch(k)\neq 2$, then there is an isomorphism of rings:
\[ Z\left(\Pi_k(k^{\oplus 4})\right) \cong \frac{k[A,B,C]}{(C^2-B(A^2-4B^2))} \]
Where $A,B$ are homogeneous elements in degree 4 and $C$ is a homogeneous element in degree 6.
\end{lemma}
\begin{proof}
By Lemma \ref{lem:dihedralskew} and the fact that the center of a ring is invariant under Morita equivalence, we have
\[ Z\left(\Pi_k(k^{\oplus 4})\right) \cong Z \big( k[x,y] \# BD_2 \big) = k[x,y]^{BD_2} \]
Hence we need to find the invariants in $k[x,y]$ under the action of 
$BD_2$ where the generators $a$ and $b$ act on $(x,y)$ through the rule
\[ a \mapsto \begin{bmatrix}  0&1\\
-1 &0 \end{bmatrix} \textrm{ and } b \mapsto \begin{bmatrix}  \xi_4 & 0\\
0 & \xi_4^3 \end{bmatrix} \]
where $\xi_4$ is a primitive 4th root of unity, i.e.
\[ a \cdot x^m y^n = (-1)^n x^n y^m \ \ \ \textrm{ and } \ \ \ b \cdot x^my^n = \xi_4^{m+3n} x^m y^n \]
Let $P(x,y) = \sum_{m,n} c_{m,n} x^m y^n \in k[x,y]$. Then $a \cdot P = P$ implies $c_{n,m} = (-1)^m c_{m,n}$ for all $m,n$ and $b \cdot P = P$ implies $c_{m,n} = 0$ unless $m + 3n \equiv 0 \ (mod \ 4)$. In particular $k[x,y] \# BD_2$ the following $k$-basis:
\begin{align*} & \{ x^{4i-2j}y^{2j} + x^{2j}y^{4i-2j} \mid i,j \in \mathbb{N}, j \leq i \} \\
 \cup & \ \{ x^{4i-2j+1}y^{2j+1} - x^{2j+1}y^{4i-2j+1} \mid i,j \in \mathbb{N}, j < i \} \end{align*}
Hence, as a $k$-algebra it is generated by $A=x^4 + y^4$, $B=x^2y^2$, $C=x^5y-xy^5$ which satisfy the relation $C^2-B(A^2-4B^2)=0$ in degree 12. We leave it to the reader to check that this is the only possible relation 
\end{proof}


\begin{corollary} \label{cor:dihedralcenter}
Let $k$ be an algebraically closed field (possibly of characteristic 2), then 
\[ \dim_k \left (Z_d \left(k,k^{\oplus 4}\right) \right) = \begin{cases} \frac{d}{4}+1 & \textrm{if } d \equiv 0 \ (mod \ 4) \\ \frac{d-2}{4} & \textrm{if } d \equiv 2 \ (mod \ 4) \\ 0 & \textrm{else} \end{cases} \]
\end{corollary}
\begin{proof}
Combining Lemma \ref{lem:classicpreprojective} with T. Schedler's result \cite[Theorem 10.1.1.]{Schedler} we may in fact assume $k$ has characteristic different from 2. In this case the result follows from an explicit computation using the presentation exhibited in Lemma \ref{lem:dihedralcenter}.
\end{proof}


For the second specific case we have:
\begin{lemma}\label{lem:explicitcenter}
Let $k$ be a field, then
\[ Z(\Pi_k(k[s,t]/(s^2,t^2))) \cong \frac{k[A,B,C]}{(C^2)} \]
Where $A,B$ are homogeneous elements of degree 4 and $C$ is a homogeneous element of degree 6. In particular the Hilbert series of $Z(\Pi_k(k[s,t]/(s^2,t^2)))$ is the same as in Corollary \ref{cor:dihedralcenter}.
\end{lemma}
\begin{proof}
We defer the proof to section $A.2$ of the appendix. 
\end{proof}

We shall use the following lemma to compute  $\dim_k(Z_d(k,F))$ in the more general case where $F$ is a Frobenius algebra of dimension 4 over a (possibly not algebraically closed) field $k$.
\begin{lemma}\label{lem:centerinequality}
Let $F$ and $G$ be two Frobenius algebras over a field $k$ such that $\frobdef{F}{G}$. Then for each $d \in \d{N}$,
\[ \dim_k(Z_d(k,F)) \geq \dim_k(Z_d(k,G)) \]
\end{lemma}

\begin{proof}
Let $D$ be the algebra deforming $F$ to $G$ provided by Definition \ref{def:frobeniusdeformation} and denote $R=k[[u]]$, $K=k((u))$. As in  the left exact sequence (\ref{eq:center}), we write $Z_d(R,D)=\ker(\phi)$ and let $\Phi$ be the matrix corresponding to $\phi$.\\
Let $\Phi_K$ denote the same matrix with coefficients viewed in the fraction field $K$ and $\Phi_k$ denote the matrix with coefficients viewed in the residue field $k$. Then by construction, 
\[ \ker(\Phi_K)=\ker(K\tr_R \phi)=Z_d(K,K\tr_RD) \]
and
\[ \ker(\Phi_k)=\ker(k\tr_R \phi)= Z_d(k,k\tr_RD) \]
Now,
\begin{align*}
\dim_k(Z_d(k,G)) & =\dim_K(K\tr_k(Z_d(k,G))\\
& = \dim_K(Z_d(K,K\tr_k G)) \\ & =\dim_K(Z_d(K,K\tr_R D)) \\ & =\dim_K( \ker(\Phi_K))
\end{align*}
Since clearly $\dim_k(\ker(\Phi_k) \ge \dim_K (\ker(\Phi_K))$, the claim follows.
\end{proof}

\begin{lemma} \label{cor:fielddim}
Let $F$ be a Frobenius algebra of dimension $4$ over a field $k$. Then
\begin{equation} \label{eq:preprojcenterdim} \dim_k \left( Z_d \left(k,F \right) \right) = \begin{cases} \frac{d}{4}+1 & \textrm{if } d \equiv 0 \ (mod \ 4) \\ \frac{d-2}{4} & \textrm{if } d \equiv 2 \ (mod \ 4) \\ 0 & \textrm{else} \end{cases} \end{equation}
\end{lemma}

\begin{proof}
If $k$ is algebraically closed, this follows from the fact that all Frobenius algebras fit inside a directed diagram of deformations by Lemmas \ref{lem:classificationfrobenius} and \ref{lem:deformations}, together with the inequality proven above in \ref{lem:centerinequality} and the fact that the result holds for the extremal cases in the diagram  satisfy the result by \ref{lem:explicitcenter} and Corollary \ref{cor:dihedralcenter}.\\
For the general case we use  the flat base change lemma \ref{lem:flatbasechange}.
\end{proof}

\begin{lemma}\label{cor:residuebasechange}
Theorems \ref{thm:centerprojectiverank2} and \ref{lem:centersplit} hold in the case where $(R,\mathfrak{m})$ is a local domain.
\end{lemma}

\begin{proof}
Let $\phi_{R,S}$ be as the morphism defining the left exact sequence (\ref{eq:center}). Then $\phi_{R,S}$ is a morphism between free $R$-modules of finite rank and hence can be represented by a matrix $\Phi$ with respect to some chosen basis for $V := \Pi_R(S)_d$ and $W :=  (\Pi_R(S)_d^{\oplus 5}) \bigoplus (\Pi_R(S)_{d+1}^{\oplus 8})$. Let $\Phi_k$ be the matrix with coefficients viewed in the residue field $k=R/\mathfrak{m}$. Then $\Phi_k$ is a matrix representing the morphism $k \otimes \phi$ after the choice of induced $k$-basis for $k \otimes_R V$ and $k \otimes_R W$.\\
Let $r = \rk(\Phi_k)$. Then there is an invertible $r \times r$ submatrix $\Psi_k$ in $\Phi_k$. The corresponding submatrix $\Psi$ of $\Phi$ has a determinant which does not lie in $\f{m}$ and is thus itself invertible. By a suitable change of basis in $V$ and $W$ we can now rewrite $\Phi$ in the following form:
\[ \Phi =
  \begin{blockarray}{crrc}
    \begin{block}{[cr|rc]}
      & \Id_{r \times r} & 0 & \\ \BAhhline{~------~}
      & 0 & \Psi' & \\
    \end{block}
  \end{blockarray}
\]
where all entries of the the submatrix $\Psi'$ lie in $\f{m}$ (any entry not in $\f{m}$ would give rise to an invertible submatrix of rank $r+1$ by elementary row and column operations). This implies that we can decompose $V$ and $W$ as a direct sum of free submodules $V = V_1 \oplus V_2$ and $W= W_1 \oplus W_2$ such that $\phi := \phi_1 \oplus \phi_2$ where $\phi_1: V_1 \stackrel{\cong}{\mor} W_1$ and $\phi_2:V_2 \mor W_2$ satisfies $k \otimes_R \phi_2' =0$. In particular,

\begin{equation}\label{eq:centerresiduefield}
Z_d(k,k \tr_R S) = \ker( k \otimes \phi) = k \otimes V_2
\end{equation}
 
and hence $V_2$ is free of  rank given by (\ref{eq:preprojcenterdim}).\\
Now, we let $K$ denote the fraction field of $R$. Then since by construction $\ker(\phi) \subset V_2$, we obtain  $K\tr_R \ker(\phi) \subset K\tr V_2$. Hence since $K$ is flat over $R$, Lemma \ref{lem:basechange} gives: 
\[ \dim_K (K\tr_R \ker(\phi))= \dim_K(K\tr Z_d(R,S))= \dim_K(Z_d(K,K \tr_R S)) \]
Which by Lemma \ref{cor:fielddim} and the above equality (\ref{eq:centerresiduefield}) is equal to  $\dim_K(K \otimes V_2)$. It follows that $\ker(\phi)=V_2$ from which we infer that $\phi_2=0$ and hence the monomorphism $Z_d(R,S) \mono \Pi_R(S)_d$ splits. It follows that $Z_d(R,S)$ is projective and since the ranks can be computed after tensoring with a field, they must be given by (\ref{eq:preprojcenterdim}).
\end{proof}

We can now finish the proofs of the main results of this section. This is done in a way similar to the proof of Theorem \ref{thm:free}:

\begin{proof}[Proof of Theorem \ref{lem:centersplit}] 
Let $d \in \mathbb{N}$ and let $\iota_{R,S}$ denote the embedding 
\[\iota_{R,S}: Z_d(R,S) \mono \Pi_R(S)_d\]
By Lemma \ref{cor:residuebasechange} we already know that the result holds if $R$ is a local domain and by the local nature of splitting (see for example \cite[Exercise 4.13, p.105]{lam}) it extends to the case where $R$ is any domain.\\[\medskipamount]
Next let $R$ be a local ring with algebraically closed residue field. Then by Lemma \ref{lem:reducetodomain} $S/R$ can be obtained as a base change of $\tilde{S}/\tilde{R}$ through a morphism $\tilde{R} \mor R$ where $\tilde{R}$ is a domain. The result follows in this case as a split embedding remains split after base change.\\[\medskipamount] 
Next, we assume $R$ is any local ring. In this case we can consider the faithfully flat morphism $R \mor \overline{R}$ provided by Lemma \ref{lem:localclosure}. As the residue field of $\overline{R}$ is algebraically closed the monomorphism $\iota_{\overline{R}, S \otimes \overline{R}} = \iota_{R,S} \otimes \overline{R}$ is split by the above case. This implies that $\iota_{R,S}$ must be split itself by Lemma \ref{lem:ffsplit} below. \\[\medskipamount]
Finally, the result follows for any ring by using the local nature of splitting \cite[Exercise 4.13, p.105]{lam}) once again.
\end{proof}

\begin{proof}[Proof of Theorem \ref{lem:basechangecenterrank4}]
This is an immediate consequence of Theorem \ref{lem:centersplit} and the fact that the construction of $\phi_{R,S}$ in (\ref{eq:center4}) is compatible with base change.
\end{proof}

\begin{proof}[Proof of Theorem \ref{thm:centerprojectiverank2}]

First let $(R,\f{m})$ be a local domain with residue field $k$ and field of fractions $K$. Then by Lemma \ref{cor:fielddim}, 
\begin{align*}
\dim_K(K\tr_R(Z_d(R,S))=&\dim_K(Z_d(K,K\tr_R S))\\
=&\dim_k(Z_d(k,k\tr_R S)) \\
=&\dim_k(k\tr_R Z_d(R,S))
\end{align*}

Hence by \cite[Chapitre 1, Corollaire 4.4]{GrothendieckSGA1}, $Z_d(R,S)$ is free of the required rank.\\[\medskipamount]
If $R$ is a domain, then for any $\f{p} \in \Spec(R)$, $R_{\f{p}}$ is a local domain such that $R_{\f{p}}\tr_R Z_d(R,S)=Z_d(R_{\f{p}}, R_{\f{p}}\tr_R S)$ is a free module of the desired rank. Serre's theorem then proves that $Z_d(R,S)$ is projective of the desired rank.\\[\medskipamount]
Next let $R$ be a local ring with algebraically closed residue field and let $\tilde{S}/\tilde{R}$ the relative Frobenius ring provided by Lemma \ref{lem:reducetodomain}. Then we know that $Z_d( \tilde{R}, \tilde{S})$ is projective over $\tilde{R}$ of the desired rank. Hence $Z_d(R,S) = Z_d( \tilde{R},\tilde{S}) \otimes R$ is free of the required rank over $R$.\\[\medskipamount]
To extend the statement to general local rings we just apply Lemma \ref{lem:localclosure}.\\[\medskipamount] 
Finally Serre's theorem extends the statement to general rings.
\end{proof}

\ref{lem:centersplit}.
\begin{lemma} \label{lem:ffsplit}
Let $(R,\f{m})$ be a local ring and let $R \mor \overline{R}$ be as in Lemma \ref{lem:localclosure}. Let $\iota: A \mono B$ be an embedding of finitely generated $R$-modules where $B$ is a projective $R$-module. If $\iota \otimes \overline{R}: A \otimes \overline{R} \mono B \otimes \overline{R}$ is a split embedding, then $\iota$ is itself a split embedding
\end{lemma}
\begin{proof}
Let $k$ be the residue field of $R$ and $\overline{k}$ its algebraic closure, then there is a commutative diagram
\[ \begin{tikzpicture}
\matrix(m)[matrix of math nodes,
row sep=3em, column sep=3em,
text height=1.5ex, text depth=0.25ex]
{ R & k \\
\overline{R} & \overline{k} \\};
\path[->,font=\scriptsize]
(m-1-1) edge (m-1-2)
        edge (m-2-1)
(m-2-1) edge (m-2-2)
(m-1-2) edge (m-2-2)
;
\end{tikzpicture} \]
As $\iota \otimes \overline{R}$ is split, $\iota \otimes \overline{k}$ is a monomorphism. The above commutative diagram (and the faithfully flatness of $k \mor \overline{k}$) implies $\iota \otimes k$ is a monomorphism. Let $C = \coker(\iota)$, then we have a long exact sequence
\[ \ldots \mor \Tor_1^R(B,k) \mor \Tor_1^R(C,k) \mor A \otimes k \stackrel{\iota \otimes k}{\mor} B \otimes k \mor C \otimes k \mor 0 \]
As $B$ is a projective $R$-module it is also flat, implying $\Tor_1^R(B,k)=0$. From this it follows that $\Tor_1^R(C,k)=0$ and hence $C$ is of projective dimension 0, and itself projective. It follows that the exact sequence $0 \mor A \stackrel{\iota}{\mor} B \mor C \mor 0$ is split.
\end{proof}

\section{$\Pi_R(S)$ is noetherian and finite over its center.}
Throughout this section, we keep the standing assumptions of \S 4, namely $S$ relative Frobenius of rank 4 over the noetherian ring $R$.\\[\medskipamount] 

The main result of the section is the following:
\begin{theorem}\label{thm:main}
$\Pi_R(S)$ is noetherian and finite over its center.
\end{theorem}
As part of the proof, we construct a morphism 
\[
\sigma_{R,S}: R[Z_4(R,S)]^{\oplus m} \mor \Pi_R(S)
\] as follows: first choose an $R$-basis $(x,y,z,w)$ for $S$ and let $e$ be the element corresponding to $1_S \in N$ and $f$ be the element corresponding to $1_S \in M$ (where we used the notation from the discussion preceding Definition \ref{def}). This choice yields an obvious morphism
\[ \pi : R < x,y,z,w,e,f> \mor T_{R \oplus S}(M \oplus N)\]
Where $x,y,z,w$ have degree 0 and $e,f$ have degree 1 in $R< x,y,z,w,e,f>$. \\The $R$-module $T(R,S)_0=R\ds S$ is generated by $(1_R, x,y,z,w)$ and these 5 elements are the images under $\pi$ of the corresponding elements in $R<x,y,z,w,e,f>$ showing that $\pi$ is surjective in degree $0$. Moreover, $T(R,S)_1=M\ds N={}_R S_S\ds {}_S S_R$ is generated by $(xe,ye,ze,we,fx,fy,fz,fw)$ as an $R$-$R$-bimodule showing that $\pi$ is also surjective in degree $1$.\\ Finally since $T(R,S)$ is generated in degree 0 and 1, it follows that $\pi$ itself is surjective. Composing with the canonical quotient map $T(R,S) \epi \Pi_R(S)$ yields a surjection
\[ \chi : R< x,y,z,w,e,f> \epi \Pi_R(S) \]

Using the map $\chi$ we will construct a finite set of generators for $\Pi_R(S)$ as a module over its center. For this notice that the $R$-module $\Pi_R(S)_{\le 6}$ is generated by the image of the words of length at most $6$ in $\{e,f\}$. This set of words is infinite, but we can reduce it to a finite set of generators for $\Pi_R(S)_{\le 6}$ using the following observations
\begin{itemize}
\item since $\{1_R,x,y,z,w\}$ forms an $R$-basis for $\Pi_R(S)_0$, we can assume that any subword of degree zero is precisely a letter in this set
\item by the definition of the multiplication of $\Pi_R(S)$, we have $e^2=f^2=0$
\end{itemize}
Hence if we let $H$ be the finite set set of words in $\{x,y,z,w,e,f\}$ of length at most 6 in $\{e,f\}$ such any two instances of $x,y,z,w$ are separated by at least one $e$ or $f$, we obtain $\chi(R \cdot H)=\Pi_R(S)_{\le 6}$. Picking an order for this set
 \[ H=\{a_1, \ldots , a_m\} \]
we can define $\sigma_{R,S}$ as
\[ \sigma_{R,S} : R[Z_4(R,S)]^{\oplus m} \rightarrow \Pi_R(S): (z_i)_{i=1}^n \mapsto \sum_{i=1}^n z_i \chi(a_i) \]
We shall prove the following theorem
\begin{theorem}\label{thm:sigmamain}
$\sigma_{R,S}$ is surjective.
\end{theorem}
From this Theorem \ref{thm:main} will readily follow as $Z_4(R,S)$ is clearly finitely generated over $R$. We once again prove Theorem \ref{thm:sigmamain} by increasing the generality of the ring $R$.

First, show that the construction of $\sigma_{R,S}$ commutes with base change: Let $R \mor R'$ be any morphism of rings, then since  both the construction of generalized preprojective algebras and taking their center commute with base change by Theorems \ref{lem:basechangecenterrank4} and \ref{lem:basechange}, we have a diagram

\begin{eqnarray} \label{eq:diagram}
\xymatrix{
R'[Z_4(R',R'\tr_RS)]^{\ds n}\ar[rr]^{\sigma_{R',R'\tr_RS}} & &\Pi_R'(R' \tr_R S)\\
R'\tr_R R[Z_4(R,S)]^{\ds n}\ar[rr]_{R'\tr_R(\sigma_{R,S})} \ar[u]_\simeq^{\zeta}& &R'\tr_R\Pi_R(S) \ar[u]_\simeq^{\eta}
}
\end{eqnarray}
where the vertical maps are isomorphisms
\begin{lemma}\label{lem:basechangesigma}
For any morphism $\varphi: R\mor R'$, the diagram in (\ref{eq:diagram}) is commutative.
\end{lemma}

\begin{proof}
Let $S/R$ be relative Frobenius. Let basis $e_1, \ldots , e_n$ be an $R$-basis for $S$ and $\lambda$ a generator for the $S$-module $\Hom_R(S,R)$. Then as in the proof of Lemma \ref{lem:basechange}, $(R' \otimes S)/R'$ is relative Frobenius with basis $1_{R'} \otimes e_1, \ldots , 1_{R'} \otimes e_n$ and generator $1_{R'} \otimes \lambda$. Following the successive steps in the construction of $\sigma_{R',R' \otimes_R S}$ outlined in the discussion preceding \ref{thm:sigmamain} we observe that
\[ \begin{cases}
\Pi_{R'}(R' \otimes_R S) &= 1_{R'} \otimes \Pi_R(S) \\
\chi_{R'} &= 1_{R'} \otimes \chi_R \\
H_{R'} &= 1_{R'} \otimes H_R
\end{cases}\]
Let $z_i$ be an element in $R[Z_4(R,S)]$ considered as the $i$th component of $R[Z_4(R,S)]^{\oplus m}$, then
\begin{eqnarray*}
\eta \circ \left(1_{R'}\tr_R(\sigma_{R,S})\right) (r' \otimes z_i ) & = & \eta \left( r' \otimes z_i \chi_R(a_i)\right) \\
 & = & r' (1 \otimes z_i \chi_R(a_i)) \\
 & = & r' (1 \otimes z_i) (1 \otimes \chi_R(a_i)) \\ 
 & = & r' (1 \otimes z_i) (\chi_{R'}(1 \otimes a_i)) \\
 & = & \sigma_{R',R' \otimes S} (r' (1 \otimes z_i)) \\
 & = & \sigma_{R',R' \otimes S} \circ \zeta (r' \otimes z_i) \qedhere
\end{eqnarray*}
\end{proof}

\begin{lemma}\label{lem:defosigma}
Let $F$ and $G$ be Frobenius algebras over $k$ such that $\frobdef{F}{G}$.\\
If $\sigma_{k,F}$ is surjective, then so is $\sigma_{k,G}$
\end{lemma}
\begin{proof}
Let $D$ be the algebra deforming $F$ to $G$ provided by Definition \ref{def:frobeniusdeformation} and write $R:=k[[u]]$, $K:=k((u))$.\\
 Then by Lemma \ref{lem:basechangesigma}, $k\tr_R \sigma_{R,D} = \sigma_{k,F}$ and Nakayama's lemma implies that $\sigma_{R,D}$ is surjective whenever $\sigma_{k,F}$ is. A second application of Lemma \ref{lem:basechangesigma} shows that $K\tr_k \sigma_{k,G} = K\tr_R \sigma_{R, D}$, showing that $K\tr_k \sigma_{k,G}$ is also surjective. Finally, $\sigma_{k,G}$  must be surjective as $K$ is faithfully flat over $k$.
\end{proof}
We will also need to establish the result in the following specific case: 
\begin{lemma}\label{lem:surjectivespecific}
Let $F:=k[s,t]/(s^2,t^2)$. Then the map $\sigma_{k,F}$ is surjective
\end{lemma}

\begin{proof}
This is proven in $A.3$.
\end{proof}

\begin{corollary} \label{cor:fieldsurjective}
Let $F$ be Frobenius over a field $k$ . Then $\sigma_{k,F}$ is surjective
\end{corollary}
\begin{proof}
If $k$ is algebraically closed, then any Frobenius algebra $F$ over $k$ can be obtained from $\displaystyle k[s,t]/(s^2,t^2)$ by a finite number of Frobenius deformations following the diagram  (\ref{eq:deformat}). Hence
the result follows from a combination of Lemma \ref{lem:surjectivespecific} and Lemma \ref{lem:defosigma}.\\
For a general field we use that $\overline{k}$ is faithfully flat over $k$.
\end{proof}

\begin{proof}[Proof of Theorem \ref{thm:sigmamain}]
If $R$ is a local ring, then $k \tr_R \sigma_{R,S} \cong \sigma_{k,k\tr_R S}$ and the result follows by the Corollary \ref{cor:fieldsurjective} and Nakayama's lemma.\\
If $R$ is a non-local ring, for any $\f{p} \in \Spec(R)$, we have $R_{\f{p}} \tr_R \sigma_{R,S}=\sigma_{ R_{\f{p}},R_{\f{p}}\tr_R S}$, which is a surjective morphism. As this holds for all $\f{p}$, $\sigma_{R,S}$ is itself surjective.
\end{proof}

\section{The global dimension of $\Pi_R(S)$}
In this section we prove the following:
\begin{theorem}\label{thm:globdim}
The global dimension of $\Pi_R(S)$ is bounded by the number \[ \max( \gldim(R), \gldim(S) ) + 2\]
\end{theorem}
We first bound the projective dimension of $R$ and $S$ as $\Pi_R(S)$-modules. 
\begin{lemma} \label{lem:relexactsequence}
The $R$-module $R \oplus S$ admits a projective resolution of the following form:
\begin{equation}
 0 \mor \Pi_R(S)(-2) \stackrel{\alpha_2}{\mor} ({}_SS_R \oplus {}_RS_S) \otimes \Pi_R(S)(-1) \stackrel{\alpha_1}{\mor} \Pi_R(S) \stackrel{\alpha_0}{\mor} R \oplus S \mor 0 \label{eq:relexact}\end{equation}
\end{lemma}
\begin{proof}
$\alpha_0$ is the canonical projection with kernel $\Pi_R(S)_{\geq 1}$. This module is generated by $\Pi_R(S)_1 = {}_SS_R \oplus {}_RS_S$, hence $\im(\alpha_1)=\ker(\alpha_0)$. Since the relations of $\Pi_R(S)$ are generated in degree 2, we also have $\im(\alpha_2) = \ker(\alpha_1)$. Hence the injectivity of $\alpha_2$ is the only nontrivial part of the claim.\\
The sequence  (\ref{eq:relexact}) is a direct sum of the following two subsequences:
\begin{equation}
 0 \mor 1_R \cdot \Pi_R(S)(-2) \mor 1_S \cdot \Pi_R(S)(-1) \mor 1_R \cdot \Pi_R(S) \mor R \mor 0 \label{eq:relexactr} \end{equation}
\begin{equation}
 0 \mor 1_S \cdot \Pi_R(S)(-2) \mor (1_R \cdot \Pi_R(S)(-1))^{\oplus 4} \mor 1_S \cdot \Pi_R(S) \mor S \mor 0 \label{eq:relexacts} \end{equation}
By Lemma \ref{lem:basechange} exactness can be checked after localization at each prime ideal of $R$ and we thus may assume that all terms in (\ref{eq:relexactr}) and (\ref{eq:relexacts}) are free $R$-modules of finite rank in each degree by Lemma \ref{lem:splithilbertseries}. The claim reduces to the following relation on the Hilbert series:  for each $d \in \mathbb{N}$ we must have
\begin{eqnarray*}
h_{d-2}(1_R \cdot \Pi_R(S)(-2)) - h_{d-1}(1_S \cdot \Pi_R(S)(-1)) + h_d(1_R \cdot \Pi_R(S)) - \delta_{d0} & = & 0 \\
h_{d-2}(1_S \cdot \Pi_R(S)(-2)) - 4 h_{d-1}(1_R \cdot \Pi_R(S)(-1)) + h_d(1_R \cdot \Pi_R(S)) - 4 \delta_{d0} & = & 0
\end{eqnarray*}
(where $h_d(-)$ denotes the rank of the degree $d$-part as an $R$-module)\\
Using Lemma \ref{lem:splithilbertseries} we see that these relations are satisfied.
\end{proof}


\begin{lemma} \label{lem:simplemodules2}
A $\Pi_R(S)$-module is simple if and only if it is simple over $R$ or simple over $S$.
\end{lemma}
\begin{proof}
A $\Pi_R(S)$-module which is simple as an $R$- or $S$-module is clearly simple as a $\Pi_R(S)$-module. 
Conversely if $M$ is a simple $\Pi_R(S)$-module, then $M = 1_RM$ or $M = 1_SM$ since $M = 1_RM \oplus 1_SM$. 
Moreover we claim that $\Pi_R(S)_{\geq 1} M = 0$ or equivalently $\Pi_R(S)_1 M =0$. For this assume (for example) that $M=1_RM$. If $x \in {}_SS_R$ then 
\[ xM = (1_S x )M = 1_S (xM) \subset 1_S M = 0 \]
and if $x \in {}_RS_S$ then
\[ xM = (x 1_S)M = x(1_SM) = 0 \]
Hence only the $R$-component in degree 0 acts non-trivially on $M$, it follows in particular that $M$ is also a simple $R$-module. The case $M=1_S M$ is completely analogous. \qedhere


\end{proof}
\begin{proof}[Proof of Theorem \ref{thm:globdim}]
By \cite[Proposition III.6.7(a)]{bass} it suffices to check that if $M$ is a simple $\Pi_R(S)$-module then: 
\[ pd_{\Pi_R(S)}(M) \leq \max( \gldim(R), \gldim(S) ) + 2 \]
By Lemma \ref{lem:simplemodules2}, $M$ is a simple $R$-module or a simple $S$-module. We assume the former, the other case being completely similar. Let $P_\bullet \mor M$ be a resolution of $M$ by projective $R$-modules of length $pd_R(M)\leq \gldim(R)$. Then for each $i$, by Lemma \ref{lem:relexactsequence} we have
\[ pd_{\Pi_R(S)}(P_i) \leq pd_{\Pi_R(S)}(R) \leq pd_{\Pi_R(S)}(R \oplus S) \leq 2\]
A standard long exact sequence-argument now gives the desired result.
\end{proof}

\newpage
\appendix
\addtocontents{toc}{\protect\setcounter{tocdepth}{1}}
\section{Explicit computations for $\displaystyle S = \frac{k[s,t]}{(s^2,t^2)}$}
We describe $\displaystyle \Pi_k (S)$ through generators and relations:
\begin{itemize}
\item $\Pi_k(S)_0 = k \oplus S$. Let $a$ denote $(1_k,0)$ and $b=(0,1_S)$ then since $a+b=1$, $a,1,s,t,st$ is a $k$-basis for $\Pi_k(S)_0$. It is clear that this set satisfies the relations 
\[ a^2=a, as=sa=at=ta=0\]
\item $\Pi_k (S)_1 = {}_kS_S \oplus {}_SS_k$. Let $f$ be $(1_S,0)$ and $e=(0,1_S)$, then we can write $\Pi_k(S)_1 = fS \oplus Se$. Hence $f,fs,ft,fst,e,se,te,ste$ is a $k$-basis for $\Pi_k(S)_1$. By construction, each generator $\neq 1$ of $\Pi_k(S)_0$ acts nontrivially on exactly one side of each component. Hence we have the relations
\[ ea=e,af=f,ae=fa=0,es=et=sf=tf=0 \]
Note that this implies $e^2=f^2=0$ since for example 
\[ e^2 = (ea) e = e (ae) = 0\]
\item It is clear that the relation $1 \otimes 1 \in {}_kS_S \otimes {}_SS_k$ takes the form $fe=0$. To compute the second relation, note that projection onto $kst$ provides the duality isomorphism $\Hom_R(S,R) \cong S$ (see Lemma \ref{lem:deformations}). It immediately follows that $(e,se,te,ste)$ is dual to $(fst, ft,fs,f)$ in the sense of Definition \ref{def}. The relation now takes the form
\begin{equation} \label{eq:useful} efst + seft + tefs + stef = 0 \end{equation}
\end{itemize}
To sumarize $\Pi_k(S)$ is a quotient of the free algebra $k<a,s,t,e,f>$ by the relations
\[ \begin{cases}
s^2=t^2=st-ts=0 \\
a^2=a,as=sa=at=ta=0 \\
ea=e,af=f,ae=fa=0,es=et=sf=tf=0 \\
fe = efst + seft + tefs + stef = 0 
\end{cases} \]
Note that $\Pi_k(S)$ is a graded algebra via $\deg(a)=\deg(s)=\deg(t)=0$ and $\deg(e)=\deg(f)=1$.
\subsection{Proof of Lemma \ref{lem:bikwaddegree}}
In this subsection we give sets of generators in each degree, hence giving an upper bound for $\dim_k \left( \Pi_k(S)_d \right)$. More explicitly we prove that
\[ \dim_k \left( \Pi_k \left(\frac{k[s,t]}{(s^2,t^2)} \right)_d \right) \leq \left\{ \begin{array}{cl} 5(d+1) & \textrm{if $d$ is even}\\4(d+1) & \textrm{if $d$ is odd} \end{array} \right. \]
For this we make the following remarks:
\begin{itemize}
\item In each degree there are generators of two types:
\begin{itemize}
\item[Type I)] Elements of the form $f * ef * ef \ldots *ef*e(f(*))$ where each $*$ is either $s,t$ or $st$
\item[Type II)] Elements of the form $(*) ef * ef  \ldots *ef*e(f(*))$ where each $*$ is either $s,t$ or $st$
\end{itemize}
\item Let $\r{R}$ denote the relation (\ref{eq:useful}), then $f\r{R}e, t\r{R}, s\r{R}, st\r{R}$ take the form
\begin{eqnarray}
\label{eq:commute1} fsefte & = & -f t efse\\
\label{eq:commute2} steft & = & - t efst\\
\label{eq:commute3} stefs & = & - s efst\\
\label{eq:commute4} stefst & = & 0
\end{eqnarray}
\item As a consequence of the above equalities, we know that for any non-zero element there is at most one appearance of $st$. For example: 
\[ f \underline{st} ef s ef \underline{st} = fstef (sefst) = - fstef (stefs) = -f(stefst)efs = 0 \]

\end{itemize}
We say any of the above elements is of bidegree $(m,n)$ if there are $m$ appearances of $s$ and $n$ appearances of $t$. It is easy to see that the above relations do not violate this bidegree and that it turns $\Pi_k(S)$ into a $\d{Z} \times \d{Z}$-graded ring.
Using the above remarks we create (minimal) sets of generators by a case-by-case study:
\begin{itemize}
\item \underline{Case 1: $d$ even and Type I}\\
All words in this case take the form $(f*e) \ldots (f*e)$. We can use relations (\ref{eq:commute1}), (\ref{eq:commute2}), (\ref{eq:commute3}) to write the element in the form $\pm (fse)^i(fste)^\varepsilon(fte)^j$ where $\varepsilon = 0,1$. For $\varepsilon = 0$ we have $\frac{d}{2}+1$ choices for $i$ and $j$ and for $\varepsilon = 1$ we have $\frac{d}{2}$ choices, giving a total $d+1$ generators.
\item \underline{Case 2: $d$ even and Type II}\\
These are elements of the form $(*) (ef*) \ldots (ef*)ef (*)$ and since there is at most one occurrence of $st$ the bidegree satisfies $ \frac{d}{2} -1\leq m+n \leq \frac{d}{2} +2$.\\
If $ m+n = \frac{d}{2} - 1$ the element can be written in the form $\pm (efs)^m(eft)^nef$, giving $ \frac{d}{2}$ choices. Similarly if $ m+n = \frac{d}{2} +2$ the element can be written in the form $\pm (sef)^{m-1} st (eft)^{n-1}$. Giving $ \frac{d}{2}+1$ choices.\\ 
Assume $ m+n = \frac{d}{2}$. If $(m,n)=(\frac{d}{2},0)$ (or $(m,n)=(0,\frac{d}{2})$) we have 2 generators: $(sef)^\frac{d}{2}$ and $(efs)^\frac{d}{2}$ (or $(tef)^\frac{d}{2}$ and $(eft)^\frac{d}{2}$).\\
In all other cases we need 3 generators: $(sef)^m(tef)^n$, $(efs)^m(eft)^n$ and $(efs)^{m-1}efstef(tef)^{n-1}$. This gives a total of $\frac{3d}{2}+1$ generators for this subcase.\\
Finally assume $m+n = \frac{d}{2} + 1$. If $(m,n)=(\frac{d}{2}+1,0)$ (or $(m,n)=(0,\frac{d}{2}+1)$) we have 1 generator: $(sef)^\frac{d}{2}s$ (or $(tef)^\frac{d}{2}t$ ).\\
In all other cases we need 3 generators: $(sef)^m(tef)^{n-1}t$, $(efs)^{m-1} efst (eft)^{n-1}$ and $(sef)^{m-1} stef (tef)^{n-1}$. This gives a total of $\frac{3d}{2}+1$ generators for this subcase.\\
For case 2 this results in $\displaystyle \frac{d}{2} + \left( \frac{3d}{2}+1 \right) + \left( \frac{3d}{2} + 2 \right) + \left( \frac{d}{2}  + 1\right) = 4(d+1)$ generators. \\
Finally adding up the number of generators from Case 1 and Case 2 yields $5(d+1)$ generators in case $d$ is even.
\item \underline{Case 3: $d$ odd and Type I}\\
All elements in this case take the form $(f*e)(f*e) \ldots (f*e)f(*)$. By a completely similar argument as above, we conclude that generators can be chosen of the following forms:\\
$(fse)^m(fte)^nf$, $(fse)^m(fte)^{n-1}ft$, $(fse)^{\frac{d-1}{2}}fs$, $(fse)^{n-1}(fte)^{m-1}fst$ and $fste(fse)^{n-1}(fte)^{m-1}f$.
This gives a total of
\[ \frac{d+1}{2} + \frac{d+1}{2} + 1 + \frac{d+1}{2} +\frac{d-1}{2} = 2(d+1) \]
generators
\item \underline{Case 4: $d$ odd and Type II}\\
Elements in this case are of the form $(*)e(f*e)(f*e) \ldots (f*e)$. Note that any such word can be obtained by taking a word from Case 3, reading it from right to left and interchanging $e$ and $f$. Applying this ``procedure'' to the generators of Case 3 yeilds a set of generators for the current case by symmetry. Hence in the current case we have $2(d+1)$ generators, adding up to $4(d+1)$ generators in case $d$ is odd.
\end{itemize}
\subsection{Proof of Lemma \ref{lem:explicitcenter}}
Consider the elements
\[ u := sef + efs + fse \textrm{ and } v:= tef + eft + fte\]
It is easy to see that $u$ is normalizing with respect to the automorphism $\sigma$ on $\Pi_k(S)$ which sends $t$ to $-t$ and is the identity on the other generators. As $\sigma^2=Id$ we have as an immediate consequence that $u^2$ is central. A completely similar discussion yields that $v^2$ is central. Using the relations in $\Pi_k(S)$ we can write $u^2$ and $v^2$ as
\[ A := sefsef + efsefs + fsefse \textrm{ and } B := teftef + efteft + ftefte\]
One then checks that the following element is central in degree 6:
\[ C = sefsteftef + efsefsteft + fsefstefte \]
The proof of the lemma now follows from several technical steps

\begin{enumerate}
\item As any nonzero word in $\Pi_k(S)$ allows at most 1 appearance of $st$ we have $C^2=0$. Any other relation in $A,B$ and $C$ can then be written as
\[ p_1(A,B) + C \cdot p_2(A,B) = 0 \]
for some polynomials $p_1, p_2$. This automatically implies $p_1=p_2=0$ (consider bidegrees!).\\
In particular there is an inclusion of rings
\[ \zeta : \frac{k[A,B,C]}{(C^2)} \hookrightarrow Z \left( \Pi_k(S) \right) \]
We will prove that this inclusion is in fact an isomorphism. We do so by checking surjectivity in each degree separately.
\item \underline{There are no homogeneous central elements of an odd degree.}\\
\underline{In particular: $\zeta$ is surjective in each odd degree.}\\
Let $x$ be a homogeneous element of odd degree. If $ex \neq 0$, then it is a linear combination of monomials starting with $f$ and hence ending in $f,fs,ft$ of $fst$, in particular such an element is never of the form $ye$. Hence the only way $ex=xe$ is possible, is when $xe=ex=0$. Similarly for $x$ to be cental we need $fx=xf=0$. We claim that a non-trivial homogeneous element $x$ of odd degree satisfying $ex=xe=fx=xf=0$ does not exist. For this let $x$ be such an element. Then $x$ is of one of the following 4 forms:
\begin{enumerate}
\item[(i)] $x = \alpha s (efs)^me + \beta (fse)^mfs$ \\ or $x = \alpha t (eft)^ne + \beta (fte)^nft$
\item[(ii)] $x = \alpha (efs)^m(eft)^n e + \beta (fse)^m(fte)^n f$
\item[(iii)] $x = \alpha (sef)^m(tef)^{n-1}te + \beta (fse)^m(fte)^{n-1} ft $ \\$+ \ \gamma (efs)^{m-1} efst (eft)^{n-1}e + \delta (fse)^{m-1} fste (fte)^{n-1} f$
\item[(iv)] $x = \alpha (sef)^{m-1}st(eft)^{n-1}e + \beta (fse)^{m-1} fst (eft)^{n-1}$
\end{enumerate}
and for each of the 4 possibilities we check that $ex=xe=fx=xf=0$ implies $\alpha = \beta (=\gamma = \delta) = 0$. However this is immediate because in each case $ex=0$ implies $\beta (= \delta) = 0$ and $xf=0$ implies $\alpha (= \gamma) = 0$.

\item In an even degree $d=2l$ there are no central elements of bidgree $(m,n)$ with $m+n = l-1$.\\
Such an element is necessarily of the form
\[ x = (efs)^m(eft)^nef \]
and does not commute with $te$ as $tex = 0 \neq xte$.

\item In an even degree $d=2l$ there are no central elements of bidgree $(m,n)$ with $m+n = l+2$.\\
Such an element is necessarily of the form
\[ x = (sefs)^{m-1}st(eft)^{n-1} \]
and does not commute with $efs$ as $efs \cdot x =0$ whereas\\ $x \cdot efs = (-1)^n (sef)^mst(eft)^{n-1} \neq 0$.

\item \underline{If the degree is $d=2l$ is even, then for each bidegree $(m,n)$ with $m+n=l$} \underline{there is one central element if $m$ and $n$ are even and no central element if}\\ \underline{one of them is odd.} \\
An element of the given degree and bidegree can be written as:
\begin{eqnarray*} x & = & \alpha (sef)^m(tef)^n + \beta (efs)^m(eft)^n \\ & & + \ \gamma (fse)^m(fte)^n + \delta (efs)^{m-1}efstef(tef)^{n-1} \end{eqnarray*}
$ex=xe$ implies $\gamma = \beta$ and $fx=xf$ implies $\alpha = \gamma$. Hence if $x$ is central it can be written as
\[ x= \alpha x_1 + \delta' x_2 \]
where
\begin{eqnarray*} x_1 & = & (sef)^m(tef)^n + (efs)^m(eft)^n \\ & & + \ (fse)^m(fte)^n + (efs)^{m-1}efstef(tef)^{n-1}\\
x_2 & = & (efs)^{m-1}efstef(tef)^{n-1}\\
\delta' & = & \delta - \alpha
\end{eqnarray*}
Now let $m = 2m' + \epsilon_m$ and $n = 2n' + \epsilon_n$ with $\epsilon_m, \epsilon_n = 0,1$, then if $A,B,u,v$ are as above one sees
\[ x_1 = u^{\epsilon_m}A^{m'}B^{n'}v^{\epsilon_n} \]
In particular $x_1$ commutes with $se$ if $\epsilon_n = 0$ and anti-commutes with $se$ if $\epsilon_n = 1$. On the other hand $se \cdot x_2 = 0 \neq x_2 \cdot se$ and $x_2 \cdot se$ is linearly independent from $x_1 \cdot se$. Hence in order for $x$ to commute with $se$ we need $\delta' = \epsilon_n = 0$. A similar argument using $te$ in stead of $se$ we see that $\epsilon_m = 0$. Hence $m,n$ and $l$ are even and setting $\alpha = 1$ gives
\[ x = A^{m'} B^{n'} \]
which is central as $A$ and $B$ are central.

\item \underline{If the degree is $d=2l$ is even, then for each bidegree $(m,n)$ with}\\ \underline{$m+n=l+1$ there is one central element if $m,n \geq 2$ are even and} \\ \underline{no central element in all other cases}. \\
An element of bidegree $(l+1,0)$ is of the form
\[ x = (sef)^ls \]
and does not commute with $e$. Similarly an element of bidegree $(0,l+1)$ does not commute with $e$, hence we can assume $m,n \geq 1$ in which case
\begin{eqnarray*} x & = & \alpha (sef)^m(tef)^{n-1}t + \beta (sef)^{m-1}stef(tef)^{n-1} \\
 & & + \ \gamma (efs)^{m-1}efst(eft)^{n-1} + \delta (fse)^{m-1}fste(fte)^{n-1} \end{eqnarray*}
$ex=xe$ implies $\gamma = \delta$ and $\alpha = 0$. $fx=xf$ implies $\beta = \delta$. Hence if $x$ is central it can (upto a scalar) be written as
\begin{eqnarray*} x & = & (sef)^{m-1}stef(tef)^{n-1} + (efs)^{m-1}efst(eft)^{n-1}\\ & & \ + (fse)^{m-1}fste(fte)^{n-1}  \end{eqnarray*}
Now let $m = 2m' + \epsilon_m$ and $n = 2n' + \epsilon_n$ and $A,B,C,u,v$ be as above, then
\[ x = u^{\epsilon_m}A^{m'-1}CB^{n'-1}v^{\epsilon_n} \]
In order for $x$ to commute with $se$ we need $\epsilon_n = 0$. A similar argument using $te$ in stead of $se$ we see that $\epsilon_m = 0$. Hence $m$ and $n$ are even such that
\[ x = A^{m'-1} C B^{n'-1} \]
which is central as $A,C$ and $B$ are central.

\item \underline{$\zeta$ is surjective in all even degrees}
As we already have injectivity of $\zeta$, it suffices to check that $\displaystyle \frac{k[A,B,C]}{(C^2)}$ and $Z \left( \Pi_k(S) \right)$ have the same dimension over $k$ in each even degree $d$. Write $d = 2l$. By the above we find:
\begin{itemize}
\item $Z_d(k,S)$ is $\displaystyle \left(\frac{l}{2}+1 \right)$ - dimensional if $l$ is even
\item $Z_d(k,S)$ is $\displaystyle \left(\frac{l-1}{2}\right)$ - dimensional if $l$ is odd
\end{itemize}
By considering the Hilbert series, one sees that this agrees with $\displaystyle \dim_k \left( \left(\frac{k[A,B,C]}{(C^2)} \right)_d \right)$. And surjectivity of $\zeta$ is proven.
\end{enumerate}

\subsection{Proof of Lemma \ref{lem:surjectivespecific}}
Let $u$ and $v$ be the normalizing elements as above. And let $V \subset \Pi_k(S)_2$ be the $k$-vector space spanned by $u$ and $v$. Let $\mu_3$ be the multiplication morphism given by the composition
\[ \mu_3: V \otimes \Pi_k(S)_1 \mor \Pi_k(S)_2 \otimes \Pi_k(S)_1 \mor \Pi_k(S)_3 \]
Then we use a brute force computation to show that $\mu_3$ must be surjective. I.e. we show that any element of $\Pi_k(S)_3$ can be written as a linear combination of elements of the form $u \cdot x$ or $v \cdot x$ with $x \in \Pi_k(S)_1$. It suffices to check this for the generators of $\Pi_k(S)_3$:
\begin{itemize}
\item[Type I)]: elements of the form $f * ef (*) $.\\
These can all be put into the form $fsef(*)$ or $ftef(*)$ where $*$ is either $s, t$ or $st$. Now use $fsef(*) = u \cdot f(*)$ and similarly $ftef(*) = v \cdot f(*)$.
\item[Type II)]: elements of the form $(*) ef * e$
\begin{itemize}
\item[$\bullet$] $efse = u \cdot e$ and $efte = v \cdot e$
\item[$\bullet$] $sefse = u \cdot se$ and $tefte = v \cdot te$
\item[$\bullet$] $sefste = u \cdot ste$ and $tefste = v \cdot ste$
\item[$\bullet$] $sefte = -t efse - eftse = v \cdot (-se)$
\item[$\bullet$] $tefse = -s efte - efste = u \cdot (-te)$
\end{itemize}
\end{itemize}
Which shows that $\mu_3$ is indeed surjective.

Now for each degree $d$ we have a commutative diagram
\begin{center}
\begin{tikzpicture}
\matrix(m)[matrix of math nodes,
row sep=3em, column sep=5em,
text height=1.5ex, text depth=0.25ex]
{ V \otimes \Pi_k(S)_{d+1} & \Pi_k(S)_{d+3} \\
V \otimes \Pi_k(S)_1 \otimes \Pi_k(S)_d & \Pi_k(S)_3 \otimes \Pi_k(S)_d \\};
\path[->,font=\scriptsize]
(m-1-1) edge (m-1-2)
(m-2-2) edge (m-1-2)
(m-2-1) edge node[below]{$\mu_3 \otimes \Pi_k(S)_d$} (m-2-2)
        edge (m-1-1);
\end{tikzpicture}
\end{center}
where the top horizontal arrow must be a surjection as the other three are surjective. Hence by induction (and the fact that $V \otimes -$ is right exact) we have for each $n \in \mathbb{N}$ a surjection
\[ \mu_{2n+\epsilon}: V^{\otimes n} \otimes \Pi_k(S)_\epsilon \mor V^{\otimes n-1} \otimes \Pi_k(S)_{2+\epsilon} \mor \ldots \mor \Pi_k(S)_{2n+\epsilon} \]
Next let $W$ be the vector space spanned by $u^2$ and $v^2$, then for each $n$ and $\omega = 1,2$ there is a surjection
\[ W^{\otimes n} \otimes V^{\otimes \omega} \twoheadrightarrow V^{\otimes 2n+\omega} \]
and we have a commutative diagram
\begin{center}
\begin{tikzpicture}
\matrix(m)[matrix of math nodes,
row sep=3em, column sep=5em,
text height=1.5ex, text depth=0.25ex]
{ W^{\otimes n} \otimes V^{\otimes \omega} \otimes \Pi_k(S)_\epsilon & W^{\otimes n} \otimes \Pi_k(S)_{2 \omega + \epsilon} \\
V^{\otimes 2n+\omega} \otimes \Pi_k(S)_\epsilon \otimes \Pi_k(S)_d & \Pi_k(S)_{4n+2\omega+\varepsilon} \\};
\path[->,font=\scriptsize]
(m-1-1) edge (m-1-2)
        edge (m-2-1)
(m-1-2) edge node[right]{$ \rho_{4n+2\omega+\epsilon}$}(m-2-2)
(m-2-1) edge node[below]{$\mu_{4n+2\omega+\varepsilon}$} (m-2-2);
\end{tikzpicture}
\end{center}
where $ \rho_{4n+2\omega+\epsilon}$ must be surjective because the other three morphisms are. Then using the commutative triangle

\begin{center}
\begin{tikzpicture}
\matrix(m)[matrix of math nodes,
row sep=3em, column sep=1em,
text height=1.5ex, text depth=0.25ex]
{W^{\otimes n} \otimes \Pi_k(S)_{2 \omega + \epsilon} & & \Pi_k(S)_{4n+2\omega+\epsilon}\\
  & k[Z_4(k,S)]_n \otimes \Pi_k(S)_{2 \omega + \epsilon} & \\};
\path[->,font=\scriptsize]
(m-1-1) edge node[auto]{$\rho_{4n+2\omega+\epsilon}$} (m-1-3)
        edge (m-2-2)
(m-2-2) edge node[below]{$\ \ \ \ \ \ \ \overline{\rho_{4n+2\omega+\epsilon}}$} (m-1-3)
;
\end{tikzpicture}
\end{center}
we must have that $\overline{\rho_{4n+2\omega+\epsilon}}: k[Z_4(k,S)]_n \otimes \Pi_k(S)_{2 \omega + \epsilon} \mor \Pi_k(S)_{4n+2\omega+\epsilon}$ must be surjective. As $2 \omega + \epsilon$ takes the values 3,4,5,6 we have an induced surjection:
\[ \overline{\rho}: k[Z_4(k,S)] \otimes \Pi_k(S)_{\leq 6} \twoheadrightarrow \Pi_k(S) \]
(where we included $\Pi_k(S)_d$ for $d=0,1,2$ on the left hand side to guarantee surjectivity in these three lowest degrees).\\
Now $\sigma_{k,S}$ factors as $\overline{\rho} \circ \varsigma$ where $\varsigma$ is the morphism:
\[ \varsigma: k[Z_4(k,S)]^{\oplus N} \mor k[Z_4(k,S)] \otimes \Pi_k(S)_{\leq 6}: (z_i)_{i=1}^N \mapsto \sum_{i=1}^N z_i \otimes \chi(a_i) \]
By the choice of the $a_i$ in $H$, $\varsigma$ is surjective and hence also $\sigma_{k,S}$ proving the lemma.
\bibliographystyle{alpha}
\bibliography{Generalized_references}

\nocite{*}

\end{document}